\documentclass{article}
\usepackage{hyperref}
\usepackage{packages} 
\newcommand{\itai}{\ \raisebox{0.6ex}{\(\cdot\)}\hspace{-0.5ex}\kern-0.3em{-}\ }
\newcommand{\half}{\dfrac{1}{2}}
\newcommand{\A}{\mathscr{A}}
\renewcommand{\l}{\ell}
\newcommand{\id}{\bar{0}}
\newcommand{\nid}{\bar{1}}
\newcommand{\fall}{\text{for all }}
\newcommand{\ex}{\text{there exists }}
\newcommand{\set}[1]{\{#1\}}

\newtheorem{definition}{Definition}[section]
\newtheorem{example}{Example}[section]

\newtheorem{property}{Property}

\newtheorem{theorem}{Theorem}[section]
\newtheorem{claim}{Claim}[section]

\newtheorem{proposition}{Proposition}[section]
\title{Continuous Algebra: Algebraic Semantics for Continuous Propositional Logic}
\author{Purbita Jana\\ Prateek Kwatra}
\date{ }

\begin{document}

\maketitle
\begin{abstract}
We present algebraic semantics for Continuous Propositional Logic (CPL) introduced by Itai Ben Yaacov, viewed as \L ukasiewicz propositional logic with a reversed truth–falsity orientation and enriched by a unary halving connective. We introduce \emph{continuous algebras} as MV-algebras together with an unary operator $\kappa$ analogous to the halving operator introduced in CPL and analyze their core structural properties, including ideals, quotient constructions, and subdirect representations. We further establish a correspondence between continuous algebras and the class of 2-divisible $\ell u$-groups, extending Mundici’s representation theory to the continuous setting. This correspondence leads to a purely algebraic proof of the weak completeness theorem for CPL.
\end{abstract}

\bigskip

\hspace{0.4cm}\textit{Keywords}: Continuous logic; Continuous algebra; \L ukasiewicz logic; MV-algebra; $\l u$-group; p-divisible group.

\bigskip

\hspace{0.4cm}Mathematics Subject Classification: 03C90, 03C66,  06D35, 03B60.

\section{Introduction}

Continuous logic offers a natural framework for reasoning about systems in which truth values vary gradually rather than discretely. Unlike classical propositional logic, where truth is binary, continuous logic allows truth values to range over a continuum, typically the unit interval $[0,1]$. A distinctive feature of this setting is the reversal of the classical truth–falsity perspective: the value $0$ represents absolute truth, while $1$ represents absolute falsity, and logical implication is oriented accordingly. This semantic reorientation is fundamental to continuous logic and underlies its treatment of graded truth.

Continuous Propositional Logic (CPL), introduced by Itai Ben Yaacov and collaborators \cite{IB, BU}, can be viewed as \L ukasiewicz propositional logic equipped with this reversed semantic orientation and enriched by an additional unary connective $\frac{1}{2}$. Semantically, the connective $\frac{1}{2}$, halves the degree of falsity of a formula. The inclusion of the halving operator into the \L ukasiewicz propositional logic ensures fullness for CPL—an analogue of functional completeness adapted to the continuous setting—of the connectives $\left\{\itai,\neg,\half\right\}$ \cite{IA}. Continuous first-order logic was subsequently developed by Ben Yaacov and Usvyatsov as a specialization of the continuous logic of Chang and Keisler \cite{CK}, further highlighting the close relationship between many-valued logics and continuous semantics.

The principal aim of this paper is to provide an algebraic semantics for propositional continuous logic and to establish an algebraic weak completeness theorem for CPL, in close analogy with Mundici’s algebraic treatment of \L ukasiewicz logic via MV-algebras. MV-algebras and their representation by lattice-ordered groups with strong unit form the backbone of completeness results. Our work extends this paradigm to CPL, taking into account both the presence of the halving connective and the reversed truth–falsity orientation inherent in continuous logic.

To this end, we introduce a new algebraic structure, called a \emph{continuous algebra}, obtained by endowing an MV-algebra with a unary operator $\kappa$ corresponding to the connective $\frac{1}{2}$. Two additional axioms regulate the interaction between $\kappa$ and the MV-operations. Although continuous algebras form a variety and many of their general properties could, in principle, be derived directly from universal algebra, we deliberately present detailed and direct proofs of their fundamental structural features in order to clarify the behavior of this newly proposed semantics. In particular, we study continuous ideals, prime ideals, homomorphisms, and quotient constructions, and we establish a subdirect representation theorem showing that every continuous algebra embeds into a subdirect product of totally ordered continuous algebras, called \emph{continuous chains}. This explicit development makes transparent how classical MV-algebraic properties extend to the continuous setting.

A central technical contribution of the paper is the identification of a close connection between continuous algebras and an appropriate class of lattice-ordered groups. Specifically, we show that continuous algebras correspond to \emph{2-divisible} $\ell u$-groups equipped with additional structure reflecting the halving operation. This correspondence plays a vital role in the main completeness argument and serves as the correct analogue of Mundici’s equivalence between MV-algebras and $\ell u$-groups. By establishing this connection, we obtain Chang completeness theorem in the continuous setting.
Building on the representation theory of continuous algebras and their correspondence with 2-divisible $\ell u$-groups, we prove an analogue of Chang’s completeness theorem for continuous algebras and use it to derive the weak completeness of Continuous Propositional Logic. In contrast to earlier approaches to continuous logic that rely on approximated strong completeness \cite{IA}, our results provide a purely algebraic route to weak completeness for CPL.

The paper is organized as follows. Section 2 recalls the necessary background on MV-algebras and Continuous Propositional Logic and introduces continuous algebras. Section 3 develops the theory of continuous ideals, homomorphisms, and quotient algebras, culminating in the subdirect representation theorem. Sections 4 and 5 investigate the correspondence between continuous algebras and 2-divisible $\ell u$-groups. Finally, Section 6 establishes the weak completeness theorem for CPL using the algebraic framework developed throughout the paper.

\section{Continuous algebra}
This section introduces continuous algebras, defined as MV-algebras equipped with a unary operation 
$\kappa$ satisfying specific axioms and inducing a lattice structure with an associated partial order. To support the development of new results, we revisit essential properties of MV-algebras and incorporate them so as to make the article self-contained. Furthermore, we study continuous propositional logic (CPL) and construct a Lindenbaum-type algebra associated with it, which naturally carries the structure of a continuous algebra.
\begin{definition}[Continuous Algebra]
A continuous algebra is a tuple $(A,\oplus,\neg,\id,\kappa)$, where $\oplus$ is a binary operation on $A$, $\id\in A$ and $\kappa,\neg$ are unary operations on $A$ such that for all $ x,y,z\in A$, 
\begin{enumerate}[label=C\arabic*.]
    \item $x\oplus(y\oplus z)=(x\oplus y)\oplus z.$\label{C1}
    \item $x\oplus y= y\oplus x$.
    \item $x\oplus \id= x$.
    \item $\neg\neg x=x$.
    \item $x\oplus \neg\id=\neg\id$.
    \item $\neg(\neg x\oplus y)\oplus y=\neg(\neg y\oplus x)\oplus x$.\label{C6}
    \item $\kappa x=x\oplus \neg\kappa x$.\label{C7}
    \item $(\neg\kappa x\oplus \kappa y)\oplus \neg(\neg x\oplus y)=\neg \id.$\label{C8}
\end{enumerate}
\end{definition}
We denote $\neg\id$ by $\nid$ throughout this paper.
Note that if $\mathscr{A}=(A,\oplus,\neg,\id,\kappa)$ is a continuous algebra, then $\mathscr{A'}=(A,\oplus,\neg,\id)$ is an MV algebra.

So, a few results about $\mathscr{A}$ are due to its MV-algebra structure, such as:
\begin{property}\label{p1}
     $x\oplus \neg x=\nid$, for all $ x\in A$.
\end{property}

    \begin{property}\label{p2} 
    $x\oplus y=\id$ implies $x, y=\id$, for all $x,y\in A$.
    \end{property}
\begin{property}\label{propdistance0iffsamepoints}
        $x=y$ if and only if $\neg(\neg x\oplus y)\oplus \neg(\neg
         y\oplus x)=\id$, for all $x,y\in A$.
    \end{property}
    \begin{property}\label{p3}
    $\neg x\oplus y=\nid$ if and only if there exists $ z\in A$ such that $x\oplus z=y$, for all $x,y\in A$.
    \end{property}
Property \ref{p3} induces a partial order on 
$\A$, analogous to the order on an MV-algebra.
\begin{definition}
 We define a relation, $\le$, on $A$ by 
\begin{equation}\label{<}
x\le y\iff \neg x\oplus y=\nid. 
\end{equation}
\end{definition}
Since this order on a continuous algebra coincides with the order on its underlying MV-algebra, the relation $\leq$ turns $A$ into a lattice, as in MV-algebras.
\begin{definition}[Continuous Chain]
    A continuous algebra is said to be a continuous chain if and only if it is totally ordered with respect to the ordering in \ref{<}.
\end{definition}
\begin{proposition}\label{pdual}
    Let $\mathscr{A}=(A,\oplus,\neg,\id,\kappa)$ be a continuous algebra. Then $\bar{\mathscr{A}}:=(A,\odot,\neg,\nid,{\kappa'})$ is also a continuous algebra where $x\odot y:=\neg(\neg x\oplus \neg y)$ and ${\kappa'}x=\neg\kappa\neg x$. 
\end{proposition}
\begin{proof}
    Let $x,y,z\in A$. Then,
    \begin{enumerate}
 \item   $x\odot (y\odot z)=\neg(\neg x\oplus (\neg y\oplus\neg z ))=\neg((\neg x\oplus \neg y)\oplus \neg z)=\neg(\neg(x\odot y)\oplus\neg z)=(x\odot y)\odot z.$
\item $x\odot y=\neg(\neg x\oplus \neg y)=\neg(\neg y\oplus \neg x)=y\odot x$. 
\item $x\odot \nid=\neg(\neg x\oplus \neg\nid)=\neg(\neg x)=x.$ 
\item $\neg\neg x=x$. 
\item $x\odot\neg\nid=\neg(\neg x\oplus \nid)=\neg\nid$. 
\item $\neg(\neg x\odot y)\odot y=\neg(\neg(x\oplus \neg y)\oplus \neg y)=\neg(\neg(y\oplus\neg x)\oplus \neg x)=\neg((\neg y\odot x)\oplus\neg x)=\neg(\neg y\odot x)\odot x.$
\item
       $ \kappa'x=\neg(\kappa\neg x)=\neg(\neg x\oplus \neg\kappa\neg x)=\neg(\neg x\oplus \kappa'x)=x\odot \neg\kappa' x.$
\item 
  $(\neg\kappa' x\odot\kappa'y)\odot \neg(\neg x\odot y)=(\kappa \neg x\odot\neg\kappa\neg y)\odot(x \oplus\neg y))=\neg((\neg \kappa \neg x\oplus \kappa\neg y)\oplus \neg(x\oplus \neg y))=\neg\nid.$
    \end{enumerate}
\end{proof}
\textbf{\underline{Remark}}:- If $(A,\oplus,\neg,\id,\kappa)$ is a continuous algebra then $\kappa'$ of $(A,\odot,\neg,\nid,\kappa')$ acts like $\delta_2$ for 2-divisible MV-algebra \cite{BG} on the underlying MV algebra $(A,\oplus,\neg,\id)$.
\begin{example}
     Consider $\A=([0,1],\oplus,\neg,0,\kappa)$ such that for all $ x,y\in [0,1]$, 
    \begin{align*}
        x\oplus y&=\min\set{1, x+y};\\
        \neg x&=1-x;\\
        \kappa x&=\frac{x+1}{2}.
    \end{align*}
    Then, $([0,1],\oplus,\neg,0)$ is an MV-algebra.
    Therefore, it remains to verify axioms \ref{C7} and \ref{C8}.
    Now for any $x,y\in [0,1]$, 
    \begin{align*}
        x\oplus\neg \kappa x&=\min\set{1,x+1-\kappa x}\\
        &=\min\left\{1,x+1-\frac{x+1}{2}\right\}\\
        &=\min\left\{1,\frac{x+1}{2}\right\}\\
        &=\frac{x+1}{2}\\
        &=\kappa x.\\
        \text{and} \quad (\neg\kappa x\oplus \kappa y)\oplus \neg(\neg x\oplus y)&=\min\set{1,1-\kappa x+\kappa y}\oplus (1-(\neg x\oplus y))\\
        &=\min\set{1,1-\kappa x+\kappa y}\oplus (1-\min\set{1,1-x+y})\\
        &=\min\set{1,\min\set{1,1-\kappa x+\kappa y}+1-\min\set{1,1-x+y} }\\
        &=\min\left\{1,\min\left\{ 1,\frac{2-x+y}{2}\right\}+1-\min\set{1,1-x+y}\right\}\\
        &=\begin{cases}
            1, &\text{if }y\ge x;\\
            \min\set{1,\frac{2-x+y}{2}+1-1+x-y}, &\text{if }y<x.
        \end{cases}\\
        &=\begin{cases}
            1, &\text{if }y\ge x;\\
            \min\set{1,1+\frac{x-y}{2}}, &\text{if }y<x.
        \end{cases}\\
        &=1.
    \end{align*}
    So, $\A$ is a continuous algebra.
    
    In the context of $[0,1]$, unless specified otherwise, the operations $\oplus$, $\neg$ and $\kappa$ are those defined in this example.
    \end{example}
\textbf{\underline{Remark}}:-For the unit interval $[0,1]$, we can define another continuous algebra using Proposition \ref{pdual}. We denote the resulting algebra, $\overline{[0,1]}$, by $[1,0]$. The operations are given by $x\odot y=\max\set{0, x+y-1}$ and $\kappa' x=\dfrac{x}{2}$.

    \begin{example}
      Consider $\A'=([0,1]\times[0,1],\oplus',\neg',(0,0),\kappa')$ such that,
    \begin{align*}
        (x,y)\oplus'(x',y')&=(x\oplus x',y\oplus y');\\
        \neg'(x,y)&=(\neg x,\neg y);\\
        \kappa'(x,y)&=(\kappa x,\kappa y).
    \end{align*}
    for all $ (x,y),(x',y')\in [0,1]\times[0,1]$.
    Then, $\A'$ is a nonlinear continuous algebra. Indeed, $$\neg'(0.4,0.6)\oplus'(0.5,0.5)=(1,0.9)\ \text{and}\ \neg'(0.5,0.5)\oplus'(0.4,0.6)=(0.9,1),\ \text{and hence}$$ neither $(0.5,0.5)\le(0.4,,0.6)$ nor $(0.4,0.6)\le(0.5,0.5)$. Therefore, the induced order is not linear.
\end{example}
Now, we present some examples of MV-algebras that are not continuous algebras.
\begin{example} Consider $\A=(\set{0,1},\vee,\neg,0)$. Then, $\A$ is a Boolean algebra and hence, an MV algebra.\\
    However, if there is some $\kappa$ on $\set{0,1}$ satisfying  axioms \ref{C7} and \ref{C8}, then 
$$\kappa0=0\vee\neg\kappa0\implies \kappa0=1-\kappa0\implies 2\kappa0=1.$$
But there is no element $x\in\set{0,1}$ such that $2x=1$.
    \end{example}
    \begin{example}
    Consider $\A=([1,2]\cap\mathds{Q},\oplus,\neg,1)$ such that for all $ x,y\in [1,2]\cap\mathds{Q}$
$$x\oplus y=\min\set{2, xy};\ \text{and}\ \neg x=\frac{2}{x}. $$
Then, $\A$ is an MV algebra. If there is some $\kappa$ on $[1,2]\cap\mathds{Q}$ satisfying axioms \ref{C7} and \ref{C8}, then 
$$\kappa1=1\oplus\neg\kappa1\implies \kappa1=\neg\kappa1\implies \kappa1=\frac{2}{\kappa1}.$$ But there is no element $x\in[1,2]\cap\mathds{Q}$ such that $x^2=2$.
\end{example}
\begin{theorem}
    The finite \L ukasiewicz chain of length $n+1$, 
    $$ \text{\L}_n=\left\{0,\frac{1}{n},\frac{2}{n},\dots,\frac{n-1}{n},1\right\},$$ is not a continuous chain for any $n\geq1$. 
\end{theorem}
\begin{proof}
    Let $n\ge 1$. If possible, let $\ex \kappa$ for \L$_n$ satisfying axioms \ref{C7} and \ref{C8} (absurdum hypothesis).\\
    Then, 
    \begin{align*}
        \kappa0=0\oplus\neg\kappa0
        &\implies \kappa0=\neg\kappa0\\
        &\implies \kappa0=1-\kappa0\\
        &\implies 2\kappa0=1\\
        &\implies \kappa0=\frac{1}{2}.
    \end{align*}
    If $\dfrac{1}{2}\in $\L$_n$, then $n$ is even.  Now, 
    \begin{align*}
        \kappa\frac{1}{n}&=\frac{1}{n}\oplus\neg\kappa\frac{1}{n}\\
        &=\min\left\{1,\frac{1}{n}+1-\kappa\frac{1}{n}\right\}.
    \end{align*}
    If $\kappa\dfrac{1}{n}=1$, then by C7, we get that $\dfrac{1}{n}=1$, but this is not true for any even $n$. So, we have
    \begin{align*}
        \kappa\frac{1}{n}=\frac{1}{n}+1-\kappa\frac{1}{n}&\implies
        2\kappa\frac{1}{n}=\frac{n+1}{n}\\
        &\implies \kappa\frac{1}{n}=\frac{n+1}{2n}\\
        &\implies 2|n+1
    \end{align*}
    This contradicts the fact that $n$ is even. So, $\kappa$ does not exist.
\end{proof}
\begin{example}
    The MV-algebra of infinitesimals \cite{CC} is not a continuous algebra. Indeed, if it were continuous, then there would exist an operation  $\kappa$ such that, $\kappa c=c\oplus \neg\kappa c$. It is clear that $\kappa c\ne0$ and $\kappa c\ne 1$. If possible, let $\kappa c=mc$ for some $m>0$. Then $\neg\kappa c=1-mc$. So, $c\oplus (1-mc)=1-(m-1)c\ne mc$. Similarly, if $\kappa c=1-mc$ for some $m>0$ then $\neg\kappa c=mc$. Then, $c\oplus mc=(m+1)c\ne 1-mc$. Therefore, there does not exist any operation $\kappa$ on the MV-algebra of infinitesimals for which it becomes a continuous algebra.
\end{example}
\begin{proposition}
    The smallest continuous subalgebra of $([0,1],\oplus,\neg,0,\kappa)$ containing $\set{0,1}$ is $$M=\left\{\dfrac{a}{2^b}\middle |  0\le a\le 2^b; a,b\in \mathbb{N}\right\}.$$
\end{proposition}
\begin{proof}
    Note that $0,1\in M$ and that $$(M,\oplus|_{M\times M},\neg|_M,0,\kappa|_M)$$ is indeed a continuous subalgebra of $([0,1],\oplus,\neg,0,\kappa)$. 

    We now claim that any superset $M'$, of $\set{0,1}$ that is closed under $\kappa$ and $\neg$ must contain $M$. We prove this by induction on $b$. For $b=0$, we have $a=0$ or $1$ and hence $$\left\{\dfrac{0}{2^0},\dfrac{1}{2^0}\right\}=\set{0,1}\subseteq M'$$ by assumption.

    Now let us assume that 
    $$\left\{\dfrac{a}{2^b}\middle|0\le a\le 2^b; b\le k; a,b\in \mathds{N}\right\}\subseteq M'.$$
    Consider the element $\dfrac{a_0}{2^{k+1}}$. If $2^k\le a_0\le 2^{k+1}$, then $\dfrac{a_0}{2^k}-1\in M'$ by induction hypothesis. Since $M'$ is closed under $\kappa$,
    $$\kappa\left (\dfrac{a_0}{2^k}-1\right)=\dfrac{1+\dfrac{a_0}{2^k}-1}{2}=\dfrac{a_0}{2^{k+1}}\in M'.$$ 
    If $0\le a_0< 2^k$, then $$1-\dfrac{a_0}{2^{k+1}}=\dfrac{a_1}{2^{k+1}}\in M'$$ for some $a_1$ such that $2^k<a_1\le 2^{k+1}$. So, $\dfrac{a_0}{2^{k+1}}\in M'$ as $M'$ is closed under $\neg$. Thus by induction $M\subseteq M'$.

    Therefore, any continuous subalgebra of $[0,1]$ containing $\{0,1\}$ must contain $M$. Since $M$ itself is a continuous subalgebra, it follows that $M$ is the smallest continuous subalgebra of 
$[0,1]$ containing $\{0,1\}$.
\end{proof}
\subsection{Continuous Propositional Logic}
We now recall some basic facts about continuous propositional logic. Let $$\mathscr{L}_0=\{P_i:i\in I\}$$ for some index set $I$. The elements $P_i\in \mathscr{L}_0$ are called propositional variables. Let $\mathscr{L}$ be the language freely generated by $\mathscr{L}_0,$ together with binary connective $\itai$ and unary connectives $\neg$ and $\half$. The logic determined by $\mathscr{L}$ is called a continuous propositional logic(CPL).
In contrast, the language, $\mathscr{L}'$, freely generated by $\mathscr{L}_0,$ together with the binary connective $\itai$ and the unary connective $\neg$ determines a \L ukasiewicz logic.
If $v_0:\mathscr{L_0}\to[0,1]$ is a mapping, then this can be extended uniquely to a mapping $$v:\mathscr{L}\to [0,1]\ \text{and}\  v':\mathscr{L}'\to[0,1]$$ defined recursively by the following clauses: 
\begin{align*}
    v'(\alpha\itai \beta)&=\max\set{0,v'(\alpha)-v'(\beta)};\\
    v'(\neg\alpha)&=1-v'(\alpha);\\
    v(\alpha\itai \beta)&=\max\set{0,v(\alpha)-v(\beta)};\\
    v(\neg\alpha)&=1-v(\alpha);\\
    v\left(\half\alpha\right)&=\half v(\alpha).
\end{align*}
The mappings $v,v'$ are called truth assignments induced by $v_0$. Let $\Sigma\subseteq \mathscr{L}\ (\text{respectively,}\ \mathscr{L}').$ Then we write $$v\vDash \Sigma\ (\text{respectively}, \ v\vDash'\Sigma)$$ if $v(\alpha)=0\ (\text{respectively},\ v'(\alpha)=0)\ \fall \alpha\in\Sigma$ and call $v$ a model of $\Sigma$. We further write $$\Sigma\vDash\alpha, (\text{respectively,}\ \Sigma\vDash'\alpha)$$ if every model of $\Sigma$ is also a model of $\set{\alpha}$. 

Continuous propositional logic (CPL) is axiomatized by the following six axiom schemata:
\begin{enumerate}[label=A\arabic*.]
    \item $(\alpha\itai \beta)\itai\alpha$.\label{A1}
    \item $((\alpha\itai\beta)\itai (\alpha\itai\gamma))\itai(\gamma\itai\beta)$.\label{A2}
    \item $(\alpha\itai(\alpha\itai\beta))\itai(\beta\itai(\beta\itai \alpha))$.\label{A3}
    \item $(\alpha\itai\beta)\itai(\neg\beta\itai\neg\alpha)$.\label{A4}
    \item $\half\alpha\itai(\alpha\itai\half\alpha)$.\label{A5}
    \item $(\alpha\itai\half\alpha)\itai\half\alpha.$ \label{A6}
\end{enumerate}
A formal deduction from $\Sigma$ is a finite sequence of formulae in $\mathscr{L},\ ( \alpha_i: i < n),$ such that for each $i < n$, either
\begin{enumerate}
    \item $\alpha_i$ is an instance of an axiom schema 
    \item or $\alpha_i\in\Sigma$  
    \item or there exist $j , k < i$ such that $\alpha_j=\alpha_i\itai\alpha_k$.
\end{enumerate}
We now define a syntactic consequence relation $\vdash$ between 
$\mathcal{P}(\mathscr{L})$ and $\mathscr{L}$ by declaring
$$\Sigma \vdash \alpha$$
if and only if there exists a formal deduction of $\alpha$ from $\Sigma$.

Similarly, we define a relation $\vdash'$ between 
$\mathcal{P}(\mathscr{L}')$ and $\mathscr{L}'$ by declaring
$$\Sigma \vdash' \alpha$$
if and only if there exists a formal deduction of $\alpha$ from $\Sigma$
in which the axiom schemata \ref{A5} and \ref{A6} are not used.

\begin{proposition}\label{p2.4}\cite{CDM}
    For any $\alpha,\beta,\gamma\in \mathscr{L}$, the following results hold:
    \begin{enumerate}[label=\theproposition.\arabic*]
        \item $\vdash ((\alpha\itai\beta)\itai\gamma)\itai((\alpha\itai\gamma)\itai\beta
        )$\label{2.4.1}
        \item $\vdash \alpha\itai\alpha$\label{2.4.2}
        \item $\vdash ((\alpha\itai\beta)\itai(\gamma\itai\beta))\itai(\alpha\itai\gamma)$\label{2.4.3}
        \item $\vdash (\alpha\itai\beta)\itai\neg\neg\alpha$\label{2.4.4}
        \item $\vdash \alpha\itai\neg\neg\alpha$\label{2.4.5}
        \item $\vdash (\neg\alpha\itai\beta)\itai(\neg\beta\itai\alpha)$\label{2.4.6}
        \item $\vdash(\neg\alpha\itai\neg\beta)\itai(\neg\neg\beta\itai\alpha)$\label{2.4.7}
        \item $\vdash \neg\neg\alpha\itai\alpha$\label{2.4.8}
        \item $\vdash (\neg\alpha\itai\neg\beta)\itai(\beta\itai\alpha)$.\label{2.4.9}
    \end{enumerate}
\end{proposition}
{\textbf{Note}:-
All these results follow by adapting the proof of Proposition 4.3.4 in \cite{CDM}, reversing the direction of the connective $\to$. We now define a relation $\equiv$ on $\mathscr{L}$ by declaring that
$$\alpha \equiv \beta 
\quad \text{if and only if} \quad
\vdash \alpha \itai \beta \ \text{and}\ \vdash \beta \itai \alpha .
$$
By Proposition \ref{p2.4}, the relation $\equiv$ is an equivalence relation.

For each $\alpha \in \mathscr{L}$, denote its equivalence class by $[\alpha]$, and let $\mathscr{L}/\!\equiv$ be the set of all equivalence classes of $\mathscr{L}$. We now define the algebra $\mathcal{C} = (\mathscr{L}/\!\equiv,\;\odot,\;\neg,\;\bot,\;\half),$
where for all $\alpha,\beta \in \mathscr{L}$ the operations are given by
\begin{align*}
[\alpha] \odot [\beta] &:= [\alpha \itai \neg \beta],\\
\neg[\alpha] &:= [\neg \alpha],\\
\bot &:= [\neg(\gamma \itai \gamma)] \quad \text{for some } \gamma \in \mathscr{L},\\
\half[\alpha] &:= \left[\half \alpha\right].
\end{align*}

\begin{theorem}
     The operations defined above on $\mathscr{L}/_{\equiv}$ are well-defined and $\mathcal{C}$ is a continuous algebra.
\end{theorem}
\begin{proof}
    Let $\alpha'\in[\alpha]$ and $\beta'\in[\beta]$.
    Then, 
    \begin{align*}
        \alpha_1&=\alpha'\itai\alpha &[\text{Hypothesis}]\\
        \alpha_2&=\beta'\itai\beta&[\text{Hypothesis}]\\
        \alpha_3&=((\alpha'\itai\neg\beta)\itai(\alpha\itai\neg\beta)) \itai(\alpha'\itai\alpha)&[\text{By Proposition } \ref{2.4.3}]\\
        \alpha_4&=(\alpha'\itai\neg\beta)\itai(\alpha\itai\neg\beta)&[\text{By MP on }\alpha_1,\alpha_3]\\
        \alpha_5&=(\neg\beta\itai\neg\beta')\itai(\beta'\itai\beta)&[\text{By Proposition } \ref{2.4.9}]\\
        \alpha_6&=\neg\beta\itai\neg\beta'&[\text{By MP on } \alpha_5,\alpha_2]\\
        \alpha_7&=((\alpha'\itai\neg\beta')\itai(\alpha'\itai\neg\beta)) \itai(\neg\beta\itai\neg\beta')&[\text{By axiom schema } \ref{A2}]\\
        \alpha_8&=((\alpha'\itai\neg\beta')\itai(\alpha'\itai\neg\beta))&[\text{By MP on }\alpha_7,\alpha_6]\\
        \alpha_9&=(((\alpha'\itai\neg\beta')\itai(\alpha\itai\neg\beta)) \itai ((\alpha'\itai\neg\beta')\itai(\alpha'\itai\neg\beta))) \itai((\alpha'\itai\neg\beta)\itai(\alpha\itai\neg\beta))&[\text{By axiom schema } \ref{A2}]\\
        \alpha_{10}&=((\alpha'\itai\neg\beta')\itai(\alpha\itai\neg\beta)) \itai ((\alpha'\itai\neg\beta')\itai(\alpha'\itai\neg\beta))&[\text{By MP on } \alpha_4,\alpha_9]\\
        \alpha_{11}&=(\alpha'\itai\neg\beta')\itai(\alpha\itai\neg\beta)&[\text{By MP on } \alpha_{10},\alpha_8]
    \end{align*}
    Hence $\{\alpha'\itai\alpha, \beta'\itai\beta\}\vdash (\alpha'\itai\neg\beta')\itai(\alpha\itai\neg\beta).$
    Similarly, other direction can be proved. 
    So, $[\alpha'\odot\beta']=[\alpha\odot\beta]$. If $[\alpha]=[\alpha']$ then $[\neg\alpha]=[\neg\alpha']$ is due to axiom schema \ref{A4} and Proposition \ref{2.4.9}.
    Now, if $\alpha,\beta\in \mathscr{L}$, then $[\alpha\itai\alpha]=[\beta\itai\beta]$ is due to Proposition $\ref{2.4.2}$.
    Therefore, $\bot$ is also well defined.
    Now, since $\mathscr{L}'$ is a \L ukasiewicz logic, we have the finite strong completeness, as proved in \cite{TF}.
    Let $\alpha,\beta,\gamma,\delta\in \mathscr{L}'$. Then, consider $T=\set{\gamma\itai(\alpha\itai\gamma),(\alpha\itai\gamma)\itai\gamma,\delta\itai(\beta\itai\delta),(\beta\itai\delta)\itai\delta,\alpha\itai\beta, \beta\itai\alpha}$.
    Then, we have $$T\vdash'\gamma\itai\delta\iff T\vDash'\gamma\itai\delta.$$ 
    Now, for some $v:\mathscr{L}\to[0,1]$ if $v(\varphi)=0\ \fall \varphi\in T$ , then
    \begin{align*}
        v(\alpha)=v(\beta)\ \text{and}\ v(\gamma)=\max\{0,v(\alpha)-v(\gamma)\}&
        \implies v(\gamma)=v(\alpha)-v(\gamma)\\
        &\implies v(\gamma)=\frac{v(\alpha)}{2}\ \text{and}\ v(\delta)=\frac{v(\beta)}{2}\\
        &\implies v(\delta)=v(\gamma)\\
        &\implies T\vDash' \delta\itai\gamma\ \text{and}\ T\vDash'\gamma\itai\delta\\
        &\implies T\vdash'\delta\itai\gamma\ \text{and}\ T\vdash' \gamma\itai\delta.
    \end{align*}
    Let $\alpha_1,\dots,\alpha_n$ be a proof of $\delta\itai \gamma$.
    Let $[\alpha_i]^{\gamma,\delta}_{\frac{1}{2}\alpha,\frac{1}{2}\beta}$ be the well formed formula obtained by substituting each instance of $\gamma$ and $\delta$ by $\half\alpha$ and $\half\beta$ respectively in $\alpha_i$.\\
    \begin{claim}$[\alpha_1]^{\gamma,\delta}_{\frac{1}{2}\alpha,\frac{1}{2}\beta}, \dots,[\alpha_n]^{\gamma,\delta}_{\frac{1}{2}\alpha,\frac{1}{2}\beta}$ is a proof of $\half\alpha\itai\half\beta$ from $\set{\alpha\itai\beta, \beta\itai\alpha}$.
    \end{claim}
    \begin{proof}
        For each $\alpha_i$, either
        \begin{enumerate}
            \item $\alpha_i$ is an axiom, in which case $[\alpha_i]^{\gamma,\delta}_{\frac{1}{2}\alpha,\frac{1}{2}\beta}$ is also an axiom,
            \item or $\alpha_i\in T$, in which case either $[\alpha_i]^{\gamma,\delta}_{\frac{1}{2}\alpha,\frac{1}{2}\beta}\in\set{\alpha\itai\beta,\beta\itai\alpha}$ or $[\alpha_i]^{\gamma,\delta}_{\frac{1}{2}\alpha,\frac{1}{2}\beta}$ is an axiom, 
            \item or $\ex k,j<i$ such that $\alpha_k=\alpha_i\itai\alpha_j$, in which case $[\alpha_k]^{\gamma,\delta}_{\frac{1}{2}\alpha,\frac{1}{2}\beta}=[\alpha_i]^{\gamma,\delta}_{\frac{1}{2}\alpha,\frac{1}{2}\beta}\itai[\alpha_j]^{\gamma,\delta}_{\frac{1}{2}\alpha,\frac{1}{2}\beta}$.
        \end{enumerate}  
        So, $\set{\alpha\itai\beta,\beta\itai\alpha}\vdash \half\alpha\itai\half\beta$ and similarly, $\set{\alpha\itai\beta,\beta\itai\alpha}\vdash \half\beta\itai\half\alpha$
    \end{proof}
    Therefore, $\half$ is well defined on $\mathscr{L}/_{\equiv}$.
    Now, we need to prove that $\mathcal{C}=(\mathscr{L}/_{\equiv},\oplus,\neg,\bot,\half)$ is a continuous algebra.
    Since $\mathscr{L}$ follows axiom schemata \ref{A1}-\ref{A4}, $\mathcal{C}$ follows axioms  \ref{C1}-\ref{C6}. Now, for axiom \ref{C7}, let $\alpha\in\mathscr{L}$, consider the following proof
    \begin{align*}
        \alpha_1:&\neg\neg\half\alpha\itai\half\alpha&[\text{By Proposition }\ref{2.4.8}]\\
        \alpha_2:&((\alpha\itai\half\alpha)\itai(\alpha\itai\neg\neg\half\alpha))\itai(\neg\neg\half\alpha\itai\half\alpha)&[\text{By axiom schema }\ref{A2}]\\
        \alpha_3:&(\alpha\itai\half\alpha)\itai(\alpha\itai\neg\neg\half\alpha)&[\text{By MP on }\alpha_1,\alpha_2]\\
        \alpha_4:&((\half\alpha\itai(\alpha\itai\neg\neg\half\alpha))\itai((\alpha\itai\half\alpha)\itai(\alpha\itai\neg\neg\half\alpha)))\itai(\half\alpha\itai(\alpha\itai\half\alpha))&[\text{By schema}\ref{2.4.3}]\\
        \alpha_5:&\half\alpha\itai(\alpha\itai\half\alpha)&[\text{By schema}\ref{A5}]\\
        \alpha_6:&(\half\alpha\itai(\alpha\itai\neg\neg\half\alpha))\itai((\alpha\itai\half\alpha)\itai(\alpha\itai\neg\neg\half\alpha))&[\text{By MP on }\alpha_5,\alpha_4]\\
        \alpha_7:&\half\alpha\itai(\alpha\itai\neg\neg\half\alpha)&[\text{By MP on }\alpha_6,\alpha_3]
    \end{align*}
    Similarly, the other direction can also be proved.
    Now, we just have to check for axiom \ref{C8}.
    Let $\alpha,\beta,\gamma,\delta\in \mathscr{L}'$. Then, consider $T=\set{\gamma\itai(\alpha\itai\gamma),(\alpha\itai\gamma)\itai\gamma,\delta\itai(\beta\itai\delta),(\beta\itai\delta)\itai\delta}$.
    Then, we have $$T\vdash'(\neg\gamma\itai\neg\delta)\itai\neg\neg(\neg\alpha\itai\neg\beta)\iff T\vDash'(\neg\gamma\itai\neg\delta)\itai\neg\neg(\neg\alpha\itai\neg\beta).$$ 
    Now, if for some $v:\mathscr{L}\to[0,1]$, $v(\varphi)=0\ \fall \varphi\in T$ then
    \begin{align*}
        & v(\gamma)=\max\{0,v(\alpha)-v(\gamma)\}\\
        &\implies v(\gamma)=v(\alpha)-v(\gamma)\\
        &\implies v(\gamma)=\frac{v(\alpha)}{2}\ \text{and}\ v(\delta)=\frac{v(\beta)}{2}\\
        &\implies v((\neg\gamma\itai\neg\delta)\itai\neg\neg(\neg\alpha\itai\neg\beta))=\max\set{0,\max\set{0,1-v(\gamma)-1+v(\delta)}-\max\set{0,1-v(\alpha)-1+v(\beta)}}\\
        &\implies v((\neg\gamma\itai\neg\delta)\itai\neg\neg(\neg\alpha\itai\neg\beta))=\max\set{0,\max\set{0,v(\delta)-v(\gamma)}-\max\set{0,v(\beta)-v(\alpha)}}\\
        &\implies v((\neg\gamma\itai\neg\delta)\itai\neg\neg(\neg\alpha\itai\neg\beta))=\max\left\{0,\half\max\set{0,v(\beta)-v(\alpha)}-\max\set{0,v(\beta)-v(\alpha)}\right\}\\
        &\implies v((\neg\gamma\itai\neg\delta)\itai\neg\neg(\neg\alpha\itai\neg\beta))=0\\
        &\implies T\vDash' (\neg\gamma\itai\neg\delta)\itai\neg\neg(\neg\alpha\itai\neg\beta)\\
        &\implies T\vdash'(\neg\gamma\itai\neg\delta)\itai\neg\neg(\neg\alpha\itai\neg\beta)\\
    \end{align*}
    Let $\alpha_1,\dots,\alpha_n$ be a proof of $(\neg\gamma\itai\neg\delta)\itai\neg\neg(\neg\alpha\itai\neg\beta)$ in $\mathscr{L}'$.
    \begin{claim}$[\alpha_1]^{\gamma,\delta}_{\frac{1}{2}\alpha,\frac{1}{2}\beta}, \dots,[\alpha_n]^{\gamma,\delta}_{\frac{1}{2}\alpha,\frac{1}{2}\beta}$ is a proof of $\left(\neg\alpha\itai\neg\half\beta\right)\itai\neg\neg(\neg\alpha\itai\neg\beta)$.
    \end{claim}
    This is proven in a similar way as above claim.
    Then, we get that $\left [\left(\neg\half\alpha\itai\neg\half\beta\right)\itai\neg\neg\left(\neg\alpha\itai\neg\beta\right)\right]=\neg\bot$.
    So, $\mathcal{C}$ follows axiom \ref{C8} as well and, hence, it is a continuous algebra.
\end{proof} 
Note that since $v(\alpha\itai \neg \beta)=\max\set{0,v(\alpha)+v(\beta)-1}$ and $v\left(\half\alpha\right)=\dfrac{v(\alpha)}{2}$, the truth values mimic $[1,0]$ and not $[0,1]$ as continuous algebras.
\begin{proposition}
        $\vdash \alpha \iff [\alpha]=\neg\bot$.
    \end{proposition}
    \begin{proof}
        Let $\alpha\in \mathscr{L}$ such that $\vdash \alpha$.
        Then, $\vdash (\alpha\itai\neg\neg(\alpha\itai\alpha))\itai\alpha$ implies that $\vdash \alpha\itai\neg\neg(\alpha\itai\alpha)$.
        Similarly, since $\vdash \neg\neg(\alpha\itai\alpha)$, $\vdash(\neg\neg(\alpha\itai\alpha)\itai\alpha)\itai \neg\neg(\alpha\itai\alpha)$ implies $\vdash \neg\neg(\alpha\itai\alpha)\itai\alpha$.
        So, $[\alpha]=\neg\bot$.
        For the converse, if $[\alpha]=\neg\bot$, then $\vdash \alpha\itai\neg\neg(\alpha\itai\alpha)$ and $\vdash\neg\neg(\alpha\itai\alpha)$ imply $\vdash\alpha$.
    \end{proof}
\section{Ideals and Homomorphisms}
In this section, we introduced the concepts of continuous ideals, prime continuous ideals, and continuous homomorphisms in the framework of continuous algebras to define subdirect products of such algebras. Finally, we characterized continuous algebras as subdirect products of continuous chains, a result that forms the backbone of the entire article.
\begin{definition}[Continuous Ideals]
    Let $\A=(A,\oplus,\neg,\id,\kappa)$ be a continuous algebra. Then, $I\subseteq A$ is said to be a continuous ideal if $I$ is an ideal of the MV algebra $(A,\oplus,\neg,\id)$, i.e. 
    \begin{enumerate}
        \item $\id\in I$,
        \item $\fall x,y\in I$, $x\oplus y\in I$.
        \item $\fall x\in I, y\in A$ such that $y\le x$, $y\in I$.
    \end{enumerate}
\end{definition}
\begin{definition}[Prime Continuous Ideals]
    Let $\A=(A,\oplus,\neg,\id,\kappa)$. $I\subseteq A$ is said to be a prime continuous ideal of $\A$ if and only if it is a prime MV ideal, i.e. $I\ne A$ and for all $x,y\in A$, $\neg(\neg x\oplus y)\in I$ or $\neg(\neg y \oplus x)\in I$.
\end{definition}
\begin{definition}[Continuous Homomorphisms and Isomorphisms]
    Let $\A=(A,\oplus,\neg,\id,\kappa),\ \mathscr{B}=(B,\oplus',\neg',\id',\kappa')$ be two continuous algebras. Then a function $h:A\to B$ is said to be a continuous homomorphism between $\A$ and $\mathscr{B}$ if $\fall x,y\in A$
    \begin{enumerate}
        \item $h(\id)=\id'$.
        \item $h(x\oplus y)=h(x)\oplus' h(y)$
        \item $h(\neg x)=\neg'h(x)$
        \item $h(\kappa x)=\kappa'h(x)$.
    \end{enumerate}
    If a continuous homomorphisms $h$ between continuous algebras $\A$ and $\mathscr{B}$ is bijective, then $h$ is said to be a continuous isomorphism between $\A$ and $\mathscr{B}$. If such an isomorphism exists, then $\A$ and $\mathscr{B}$ are said to be isomorphic to each other. This is denoted by $\A\cong \mathscr{B}$.
\end{definition}
The results concerning continuous homomorphisms that involve the operator 
$\kappa$ are established below. The remaining results are routine and follow from the corresponding arguments in the theory of MV-algebras.
\begin{proposition}
    Kernel of a continuous homomorphism $h:A\to B$, i.e., $ker(h):= h^{-1}(\set{\id'})$ is a continuous ideal in $\A$.
\end{proposition} 
\begin{proposition}
    The image of a continuous homomorphism is also a continuous algebra.
\end{proposition} 
\begin{proposition}\label{p3.3}
    Composition of two continuous homomorphisms is also a continuous homomorphism.
\end{proposition}
\begin{proof}
    Let $\A_1=(A_1,\oplus_1,\neg_1,\id_1,\kappa_1),\ \A_2=(A_2,\oplus_2,\neg_2,\id_2,\kappa_2)$ and $\A_3=(A_3,\oplus_3,\neg_3,\id_3,\kappa_3)$ be three continuous algebras.
    Let $g:A_1\to A_2$ and $h:A_2\to A_3$ be continuous homomorphisms.
    Then, $h\circ g(\id_1)=h(\id_2)=\id_3$ and, for any $x,y\in A_1$, 
$$ h\circ g(x\oplus_1 y)=h(g(x)\oplus_2g(y))=h(g(x))\oplus_3h(g(y))=h\circ g(x)\oplus_3h\circ g(y),$$
        $$h\circ g(\neg_1x)=h(\neg_2g(x))=\neg_3h(g(x))=\neg_3h\circ g(x)\ \text{and}$$
        $$h\circ g(\kappa_1x)=h(\kappa_2g(x))=\kappa_3h(g(x))=\kappa_3h\circ g(x).$$
\end{proof}
\begin{proposition}\label{l3.1}
    If $h:A\to B$ is a homomorphism, then $ker(h)=\set{\id'}$ if and only if $h$ is injective.
\end{proposition}
Now, Let $I\subseteq A$ be an ideal. Then, we can define a relation $\equiv_I$ on $A$ by 
$$x\equiv_I y\iff \neg(\neg x\oplus y)\oplus\neg(\neg y\oplus x)\in I.$$
Then note that $\equiv_I$ is symmetric by definition .
For $x,y,z\in A$,
$$\neg(\neg x\oplus x)\oplus\neg(\neg x\oplus x)=\neg\nid\oplus\neg\nid=\id\in I.$$
Hence $x\equiv_I x$.
Now, if $x\equiv_I y$ and $y\equiv_I z$, then we have
\begin{align*}
    (\neg x\oplus z)\oplus \neg(\neg x\oplus y)\oplus \neg(\neg y\oplus z)&=\neg x\oplus \neg(x\oplus y)\oplus z\oplus \neg(\neg y\oplus z)&[\text{By axiom }\ref{C1}]\\
    &=\neg(x\oplus \neg y)\oplus \neg y\oplus y\oplus \neg(y\oplus \neg z)&[\text{By axiom }\ref{C6}]\\
    &=\nid&[\text{By Property }\ref{p1}]
    \end{align*}
    \begin{align*}
    \text{Hence}\ & \neg(\neg x\oplus z)\le \neg(\neg x\oplus y)\oplus \neg (\neg y\oplus z)\ \text{and}\ \neg(\neg z\oplus x)\le \neg(\neg y\oplus x)\oplus\neg(\neg z\oplus y)&[\text{By }\ref{<}]\\
    &\implies \neg(\neg x\oplus z)\oplus \neg(\neg z\oplus x)\le \neg(\neg x\oplus y)\oplus \neg (\neg y\oplus z)\oplus  \neg(\neg y\oplus x)\oplus\neg(\neg z\oplus y)&\\
    &\implies \neg(\neg x\oplus z)\oplus \neg(\neg z\oplus x)\in I&\\
    &\implies x\equiv_I z.&
\end{align*}
Hence $\equiv_I$ is an equivalence relation. We denote by $[x]_I$ the equivalence class of $x$ with respect to $\equiv_I$.
\begin{definition}[Quotient of a Continuous Algebra]
Let $\A=(A,\oplus,\neg,\id,\kappa)$ be a continuous algebra, and let 
$I \subseteq A$ be a continuous ideal of $\A$. The quotient of $\A$ by $I$,
denoted by $\A/I$, is defined as
$\A/I := (A/I,\oplus_I,\neg_I,[\id],\kappa_I)$,
where
$A/I := \{[x]_I \mid x \in A\}$.
The operations $\oplus_I$, $\neg_I$, and $\kappa_I$ on $A/I$ are defined by
for all $x,y \in A$:

\begin{enumerate}
    \item $[x]_I\oplus_I[y]_I=[x\oplus y]_I$
    \item $\neg_I[x]_I=[\neg x]_I$
    \item $\kappa_I[x]_I=[\kappa x]_I$.
\end{enumerate}
\end{definition}

\begin{proposition}
    The operations on $A/I$ defined above are well-defined and $\A/I$ is a continuous algebra.
\end{proposition}
\begin{proof}
    Let $[x]_I=[x']_I$ and $[y]_I=[y']_I$.
    Then, $\neg(\neg x\oplus x')\oplus\neg(\neg x'\oplus x), \neg(\neg y\oplus y')\oplus\neg(\neg y'\oplus y)\in I$ and
    \begin{align*}
         (\neg(x\oplus y)\oplus (x'\oplus y'))\oplus\neg(\neg x\oplus x')\oplus \neg(\neg y\oplus y')&=\neg(x\oplus y)\oplus x'\oplus\neg(\neg x\oplus x')\oplus y'\oplus \neg(\neg y\oplus y')&[\text{By axiom }\ref{C1}]\\
        &=\neg(x\oplus y)\oplus x\oplus\neg(\neg x'\oplus x)\oplus y\oplus\neg(\neg y'\oplus y)&[\text{By axiom }\ref{C6}]\\
        &=\neg(x\oplus y)\oplus (x\oplus y)\oplus\neg(\neg x'\oplus x)\oplus \neg(\neg y'\oplus y)&[\text{By axiom }\ref{C1}]\\
        &=\nid&[\text{By Property }\ref{p1}]
         \end{align*}
        Hence by \ref{<}, $\neg(\neg(x\oplus y)\oplus (x'\oplus y'))\le \neg(\neg x\oplus x')\oplus \neg(\neg y\oplus y')$ and $\neg(\neg(x'\oplus y')\oplus (x\oplus y))\le \neg(\neg x'\oplus x)\oplus \neg(\neg y'\oplus y)$. Therefore
        \begin{align*}
        &\neg(\neg(x\oplus y)\oplus (x'\oplus y'))\oplus \neg(\neg(x'\oplus y')\oplus (x\oplus y))\le \neg(\neg x\oplus x')\oplus \neg(\neg y\oplus y'))\oplus \neg(\neg x'\oplus x)\oplus \neg(\neg y'\oplus y)\\
        \implies& \neg(\neg(x\oplus y)\oplus (x'\oplus y'))\oplus \neg(\neg(x'\oplus y')\oplus (x\oplus y))\in I\\
        \implies& [x\oplus y]_I=[x'\oplus y']_I.
    \end{align*}
    Well-definedness for $\neg_I$ is due to symmetry of $\equiv_I$.
    Now, if $[x]_I=[x']_I$ then
    \begin{align*}
        &(\neg\kappa x\oplus \kappa x')\oplus \neg(\neg x\oplus x')=\nid\ \text{and}\
        (\neg\kappa x'\oplus \kappa x)\oplus \neg(\neg x'\oplus x)=\nid&[\text{By axiom }\ref{C8}]\\
       \implies& \neg(\neg\kappa x\oplus \kappa x')\le \neg(\neg x\oplus x')\ \text{and}\ \neg(\neg \kappa x'\oplus \kappa x)\le\neg(\neg x'\oplus x)&[\text{By equation}\ref{<}]\\
        \implies& \neg(\neg\kappa x\oplus \kappa x')\oplus \neg(\neg\kappa x'\oplus \kappa x)\le\neg(\neg x\oplus x')\oplus\neg(\neg x'\oplus x)\in I\\
        \implies& \neg(\neg\kappa x\oplus \kappa x')\oplus\neg(\neg\kappa x'\oplus \kappa x)\in I\\
        \implies& [\kappa
        x]_I=[\kappa x']_I.
    \end{align*}
    So, $\kappa_I$ is well-defined.
    Since $(A,\oplus,\neg,\id,\kappa)$ is a continuous algebra, it follows that $(A/I,\oplus_I,\neg_I,[\id],\kappa_I)$ also follows axioms \ref{C1}-\ref{C8} and hence, is a continuous algebra.
\end{proof}

\subsection{Subdirect Product}
Let $I$ be a non-empty (Index) set.
\begin{definition}
    Let $\set{\A_i}_{i\in I}$ be a family of continuous algebras. Then, product of $\set{\A_i}_{i\in I}$ is defined by $\prod_{i\in I}\A_i:=(\prod_{i\in I}A_i,\oplus,\neg,\textbf{0},\kappa)$, where 
\begin{enumerate}
    \item $\prod_{i\in I}A_i:=\set{f:I\to\bigcup_{i\in I}A_i| f(i)\in A_i \fall i\in I}$
    \item $(f\oplus g)(i):= f(i)\oplus_ig(i)$
    \item $(\neg f)(i):= \neg_i f(i)$
    \item $\textbf{0}(i):=\id_i$
    \item $(\kappa f)(i):=\kappa_i f(i)$.
\end{enumerate}
\end{definition}
\begin{proposition}
    The operations on $\prod_{i\in I}A_i$ as defined above make $\prod_{i\in I}\A_i$ a continuous algebra.
\end{proposition}
\begin{proof}
    Let $f,g,h\in \prod_{i\in I}A_i$. Then for each $j\in I$, we have the following:
    \begin{enumerate}
        \item $((f\oplus g)\oplus h)(j)=(f\oplus g)(j)\oplus_j h(j)=(f(j)\oplus_j g(j))\oplus_j h(j)=f(j)\oplus_j(g(j)\oplus_j h(j))=f(j)\oplus_j(g\oplus h)(j)=(f\oplus(g\oplus h))(j)$.
        Hence $(f\oplus g)\oplus h=f\oplus(g\oplus h).$
        \item $(f\oplus g)(j)=f(j)\oplus_jg(j)=g(j)\oplus_jf(j)=(g\oplus f)(j)$.
        Hence $f\oplus g=g\oplus f.$
        \item $(f\oplus \mathbf{0})(j)=f(j)\oplus_j\mathbf{0}(j)=f(j)\oplus \id_j=f(j).$
        Hence $f\oplus \mathbf{0}=f.$
        \item $(\neg\neg f)(j)=\neg_j(\neg f)(j)=\neg_j\neg_j f(j)=f(j).$ Hence $\neg\neg f=f.$
        \item $ (f\oplus \neg\mathbf{0})(j)=f(j)\oplus_j\neg\mathbf{0}(j)=f(j)\oplus_j\neg_j\mathbf{0}(j)=f(j)\oplus_j\neg_j\id_j=\neg_j\id_j=\neg\mathbf{0}(j).$
        Hence $f\oplus\neg\mathbf{0}=\neg\mathbf{0}.$
        \item $(\neg(\neg f\oplus g)\oplus g)(j)=\neg(\neg f\oplus g)(j)\oplus_j g(j)=\neg_j(\neg f\oplus g)(j)\oplus_jg(j)=\neg_j(\neg f(j)\oplus_jg(j))\oplus_j g(j)=\neg_j(\neg_j f(j)\oplus_j g(j))\oplus_j g(j)=\neg_j(\neg_j g(j)\oplus_j f(j))\oplus_j f(j)=\neg_j(\neg g(j)\oplus_j f(j))\oplus_j f(j)=\neg_j(\neg g\oplus f)(j)\oplus_j f(j)=\neg(\neg g\oplus f)(j)\oplus_j f(j)=(\neg(\neg g\oplus f)\oplus f)(j).$
        Hence $\neg(\neg f\oplus g)\oplus g=\neg(\neg g\oplus f)\oplus f.$
        \item $(\kappa f)(j)=\kappa_j f(j)=f(j)\oplus_j\neg_j \kappa_j f(j)=f(j)\oplus_j\neg_j(\kappa f)(j)=f(j)\oplus_j(\neg\kappa f)(j)=(f\oplus\neg\kappa f)(j).$ Hence $\kappa f=f\oplus\neg\kappa f.$
        \item $((\neg \kappa f\oplus\kappa g)\oplus\neg (\neg f\oplus g))(j)=(\neg \kappa f\oplus\kappa g)(j)\oplus_j\neg (\neg f\oplus g)(j)=(\neg \kappa f(j)\oplus_j\kappa g(j))\oplus_j\neg_j(\neg f\oplus g)(j)=(\neg_j \kappa f(j)\oplus_j\kappa_j g(j))\oplus_j\neg_j(\neg f(j)\oplus_j g(j))=(\neg_j \kappa_j f(j)\oplus_j\kappa_j g(j))\oplus_j\neg_j(\neg_j f(j)\oplus_j g(j))=\neg_j\id_j=\neg\mathbf{0}(j).$ Hence $(\neg\kappa f\oplus \kappa g)\oplus\neg(\neg f\oplus g)=\neg\mathbf{0}.$
    \end{enumerate}
    Thus $\prod_{i\in I}\A_i$ satisfies axioms \ref{C1}-\ref{C8} and hence is a continuous algebra.
\end{proof}
\begin{definition}(Projection Map)
    Let $\set{A_i}_{i\in I}$ be a family of continuous algebras. Then, the projection map from $\prod_{i\in I}\A_i$ to $\A_j$ is defined by $\pi_j:\prod_{i\in I}A_i\to A_j$ such that $\pi_j(f)=f(j)$.
\end{definition}
\begin{proposition}
    $\pi_j$ is a continuous homomorphism for each $j\in I$.
\end{proposition}
\begin{proof}
    Let $j\in I$, then $\pi_j(\mathbf{0})=\mathbf{0}(j)=\id_j$.
    If $f,g\in \prod_{i\in I}A_i$, then 
    \begin{itemize}
        \item $ \pi_j(f\oplus g)=(f\oplus g)(j)=f(j)\oplus_j g(j)=\pi_j(f)\oplus_j\pi_j(g).$
        \item $\pi_j(\neg f)=(\neg f)(j)=\neg_jf(j)=\neg_j\pi_j(f).$
        \item $\pi_j(\kappa f)=(\kappa f)(j)=\kappa_jf(j)=\kappa_j\pi_j(f).$
    \end{itemize}
\end{proof}
We are now in a position to define subdirect products of continuous algebras.
\begin{definition}[Subdirect Product]
    A continuous algebra $\A$ is said to be a subdirect product of a family of continuous algebras $\set{\A_i}_{i\in I}$ if and only if there exists an injective continuous homomorphism $h:A\to\prod_{i\in I}A_i$ such that $\fall j\in I$, $\pi_j\circ h$ is a surjective homomorphism.
\end{definition}
\begin{proposition}
    A continuous algebra $\A$ is a subdirect product of a family of continuous algebras $\set{\A_i}_{i\in I}$ if and only if there exists a family of continuous ideals of $\A$, $\set{J_{i}}_{i\in I}$, such that
    \begin{enumerate}
        \item $\A_i\cong \A/J_i\ \ \fall i\in I$.
        \item $\bigcap_{i\in I} J_i=\set{\id}$.
    \end{enumerate}
\end{proposition}
\begin{proof}
    First, assume that $\A$ is a subdirect product of a family of continuous algebras $\set{\A_i}_{i\in I}$. Then there exists an injective continuous homomorphism $h:A\to\prod_{i\in I}A_i$ such that $\fall j\in I$ the projection $\pi_j\circ h$ is surjective.
    For each $j\in I$ define $J_j:=ker(\pi_j\circ h)$.
    Now define a map $\Phi:A/J_j\to A_j$ given by $$\Phi([x]_{J_j})=\pi_j\circ h(x).$$
    We first show that $\Phi$ is well-defined.
    Suppose $[x]_j=[y]_j$. Then we have,
    \begin{align*}
        &\neg(\neg x\oplus y)\oplus \neg(\neg y\oplus x)\in J_j (=ker(\pi_j\circ h))\\
        \implies& \pi_j\circ h(\neg(\neg x\oplus y)\oplus \neg(\neg y\oplus x))=\id_j\\
        \implies& \neg_j(\neg_j \pi_j\circ h(x)\oplus_j \pi_j\circ h(y))\oplus_j \neg_j(\neg_j \pi_j\circ h(y)\oplus_j \pi_j\circ h(x))=\id_j\\
        \implies& \neg_j(\neg_j \pi_j\circ h(x)\oplus_j \pi_j\circ h(y))=\id_j\ \text{and} \neg_j(\neg_j \pi_j\circ h(y)\oplus_j \pi_j\circ h(x))=\id_j &[\text{By Property }\ref{p2}]\\
        \implies& \neg_j \pi_j\circ h(x)\oplus_j \pi_j\circ h(y)=\nid_j\ \text{and} \neg_j \pi_j\circ h(y)\oplus_j \pi_j\circ h(x)=\nid_j\\
        \implies& \pi_j\circ h(x)\le \pi_j\circ h(y)(\le  \pi_j\circ h(x))&[\text{By equation}\ref{<}]\\
        \implies& \pi_j\circ h(x)=\pi_j\circ h(y).
    \end{align*}
    So, $\Phi$ is well-defined.
    By construction and Proposition \ref{p3.3}, $\Phi_j$ is a surjective continuous homomorphism. To show injectivity, suppose
     $\Phi([x]_j)=\Phi([y]_j)$, then 
    \begin{align*}
        &\pi_j\circ h(x)=\pi_j\circ h(y)\\
        \implies& \neg_j \pi_j\circ h(x)\oplus_j \pi_j\circ h(y)=\nid_j\ \text{and}\ \neg_j \pi_j\circ h(y)\oplus_j\pi_j\circ h(x)=\nid_j\\
        \implies& \neg_j(\neg_j \pi_j\circ h(x)\oplus_j \pi_j\circ h(y))\oplus_j\neg_j(\neg_j \pi_j\circ h(y)\oplus_j \pi_j\circ h(x))=\id_j\\
        \implies& \pi_j\circ h(\neg(\neg x\oplus y)\oplus\neg(\neg y\oplus x))=\id_j\\
        \implies& \neg(\neg x\oplus y)\oplus\neg(\neg y\oplus x)\in J_j\\
        \implies& [x]_{J_j}=[y]_{J_j}.
    \end{align*}
    So, $\Phi$ is an isomorphism.
    Now, let $x\in \bigcap_{i\in I}J_i$. Then, consider $f=h(x)\in \prod_{i\in I}A_i$. Then, for each $i\in I$, we have 
$$ f(i)=\pi_i(f)=\pi_i\circ h(x)=\id_i.$$
    Consequently $f=\textbf{0}$. But since $h$ is injective homomorphism, $ker(h)=\set{\id}$ by Proposition \ref{l3.1}. So, $x=\set{\id}$. Hence, $\bigcap_{i\in I}J_{i}=\set{\id}$.

    Conversely, suppose that for each $i\in I$ there exists an isomorphism
$\Phi_i:\mathcal A/J_i\to \mathcal A_i$, and that $\bigcap_{i\in I}J_i=\{\id\}$. Define
$$ h:A\to \prod_{i\in I}A_i
\quad\text{by}\quad
h(x)(i)=\Phi_i([x]_{J_i}).$$
    
    Now we have to show that $h$ is a continuous homomorphism. So, let $x,y\in A$. Then, for any $j\in I$ we have the following:
    \begin{enumerate}
       \item $h(x\oplus y) (j)=\Phi_j([x \oplus y]_{J_j})=\Phi_j([x]_{J_j}\oplus_{J_j} [y]_{J_j})=\Phi_j([x]_{J_j})\oplus_j\Phi([y]_{J_j})=h(x)(j)\oplus_jh(y)(j)=(h(x)\oplus_{\prod\A_{i}}h(y))(j)$
        Hence $h(x\oplus y)=h(x)\oplus_{\prod\A_i}h(y)$
        \item $h(\neg x)(j)=\Phi_j([\neg x]_{J_j})=\Phi_j(\neg_{J_j}[x]_{J_j})=\neg_j\Phi([x]_{J_j})=\neg_jh(x)(j)=(\neg_{\prod\A_i}h(x))(j)$
        Hence $h(\neg x)=\neg_{\prod\A_i}h(x)$
        \item $h(\id)(j)=\Phi_j([\id]_{J_j})=\id_j=\mathbf{0}(j)$
        Hence $h(\id)=\mathbf{0}$
        \item $h(\kappa x)(j)=\Phi_j([\kappa x]_{J_j})=\Phi_j(\kappa_{J_j}[x]_{J_j})=\kappa_j\Phi_j([x]_{J_j})=\kappa_j h(x)(j)=(\kappa_{\prod\A_i}h(x))(j)$
        Hence $h(\kappa x)=\kappa_{\prod\A_i}h(x)$.
    \end{enumerate}

Thus, $h$ is a continuous homomorphism. Since each $\Phi_j$ is surjective, $\pi_j\circ h$
is surjective for all $j\in I$.

Finally, let $x\in\ker(h)$. Then $h(x)(i)=\id_i$ for all $i\in I$, and hence
    
    \begin{align*}
       \Phi_i([x]_{J_i})=\id_i&\implies [x]_{J_i}=[\id]_{J_i}\\
        &\implies \neg(\neg x\oplus \id)\oplus\neg(\neg\id\oplus x)\in J_i\\
        &\implies \neg\neg x\oplus \neg\nid\in J_i\\
        &\implies \neg\neg x\in J_i\\
        &\implies x\in J_i\\
        &\implies x\in \bigcap_{i\in I}J_i\\
        &\implies x=\id.
    \end{align*}
    Thus, $x\in \bigcap_{i\in I}J_i=\{\id\}$, and so $\ker(h)=\{\id\}$. By
Proposition \ref{l3.1}, $h$ is injective.

Therefore, $\mathcal A$ is a subdirect product of the family $\{\mathcal A_i\}_{i\in I}$.
\end{proof}
\begin{theorem}
    Every non-trivial continuous algebra is a subdirect product of continuous chains.
\end{theorem}
\begin{proof}
    Let $\A$ be a non-trivial continuous algebra.
    Then, for each $a\in A$ such that $a\neq \id$, there exists a prime ideal $P$ of $\A$ such that $a\notin P$, by \cite{CM}. So, there exists a family of prime ideals of $A$, $\set{P_i}_{i\in I}$ such that $\bigcap_{i\in I}P_i=\set{\id}$. Also, $\A/P$ is continuous chain if $P$ is a prime ideal of $\A$, by \cite{CM}. So, $\A/P_i$ is continuous chain for each $P_i$. Therefore, by previous proposition, $\A$ is a subdirect product of a family of continuous chains $\{\A/P_i\}_{i\in I}$.
\end{proof}
\section{$2$-divisible $\l u$-groups}
We now consider a subclass of $\ell u$-groups that is categorically equivalent to the class of MV-algebras, providing a framework for studying continuous algebras.

\begin{definition}[$\l u$-group\cite{CDM}]
    An $\l u$-group is a tuple $(G,+,-, id, \vee, \wedge, u)$, where:
\begin{enumerate}
    \item $(G, +, -, id)$ is an abelian group.
    \item $(G, \vee, \wedge)$ is a lattice.
    \item For all $x, y, z \in G$:
    \[
    (x \vee y) + z = (x + z) \vee (y + z).
    \]
    \item $u$ is a strong unit, that is, for all $x \in G$, there exists a natural number $n$ such that:

    \[
    |x| \leq \underbrace{u+ u+ \dots + u}_{\text{n times}}.
    \]
\end{enumerate}
\end{definition}
\begin{definition}
    A group $(G,+)$ is said to be $2$-divisible if for each $g\in G$ there exists an $h\in G$ such that $g=h+h$.
\end{definition}

Some examples of the $2$-divisible $\l u$-groups are:
\begin{example}
    \begin{enumerate}
        \item $(\mathds{R},+,-,0,\max,\min,1)$.
        \item $(\mathds{Q},+,-,0,\max,\min,1)$.
        \item $(\mathds{R}^+, \times, ^{-1}, 1, \max,\min,2)$.
    \end{enumerate}
\end{example}
On the other hand, some examples of $\l u$-groups which are not $2$-divisible $\l u$-groups are:
{\begin{example}
    \begin{enumerate}
        \item $(\mathds{Q}_{>0},\times,^{-1},1,\max,\min,2)$ since $x^2=2$ implies $x\notin \mathds{Q}$.
         \item $(\mathds{Z}/2\mathds{Z},+,-,0,\max,\min,1)$ $x+x=1$ implies $ x\notin \mathds{Z}/2\mathds{Z}$.
    \end{enumerate}
\end{example}
\begin{proposition}\label{propositionbirkoff}
    Let $G$ be an $\l u$-group. If $a+a\ge b+b$ for some $a, b\in G$, then $a\ge b$.
\end{proposition}
In fact, this result is a special case of Proposition 3 in Chapter 14 of \cite{BG}.
\begin{proposition}\label{propdecomp}\cite{BG}
    Let $G$ be an $\l u$-group. Then, for any $g\in G$, $g=g^++g^-$ where $g^+=g\vee id$ and $g^-=g\wedge id$.
\end{proposition}
\begin{theorem}\label{TFAElu*}
    Let $G$ be an $\l u$-group. Then, the following are equivalent:
    \begin{enumerate}
        \item $G$ is a $2$-divisible $\l u$-group.
        \item There exists a function $^*:[id, u]\to G$ such that $g^*+g^*=u+g$; where $[id,u]=\set{g\in G| id\le g\le u}$.
        \item The function $f:G\to G$ defined by $f(g)=g+g$ is an automorphism.
    \end{enumerate}
\end{theorem}
\begin{proof}
\textbf{(1 $\implies$ 2)}: 
    First, assume that $G$ is a $2$-divisible $\l u$-group. Then, for each $g\in [id, u]$, consider the set $$X_g:=\set{h\in G| h+h=u+g}.$$ Since $G$ is a $2$-divisible group, $X_g$ is non-empty. So, by axiom of choice, there exists a function $$F:\set{X_g|g\in [id, u]}\to \bigcup_{g\in[id, u]} X_{g}\ \text{such that}\ F(X_g)\in X_g.$$ Then the required $^*:[id,u]\to G$ is defined by $g^*=F(X_g)$.

\textbf{(2 $\implies$ 3)}:  Let $*:[id, u]\to G$ be a function such that $g^*+g^*=u+g$. Define $f:G\to G$ by $f(g)=g+g$. We check that $f$ is an automorphism.
    \begin{enumerate}
        \item As $f(g+h)=g+h+g+h=g+g+h+h=f(g)+f(h)$, $f$ preserves group operation.
        \item As $f(-g)=-g+(-g)=-(g+g)=-f(g)$, $f$ preserves inverses.
        \item Let $g\le h$. Then, $f(g)=g+g\le g+h\le h+h=f(h)$. Conversely, if $f(g)\ge f(h)$, then $g+g\ge h+h$, which implies $g\ge h$, using Proposition \ref{propositionbirkoff}. So, $f$ preserves order.
        \item Let $f(g)= f(h)$. Then, $g+g= h+h$. So, $g+g\ge h+h$ implies $g\ge h$ and $h+h\ge g+g$ implies $h\ge g$. So, $g=h$. Therefore, $f$ is injective.
        \item For surjectivity, let $g\in G$. If $g=id$, then $g=id+id=f(id)$. Otherwise, exactly one of the following cases occurs:
        \begin{enumerate}
            \item \( g > id \),
            \item \( g < id \),
            \item \( g \) and \( id \) are incomparable, which we denote by \( g \not\sim id \).
        \end{enumerate}
       \textbf{Case (a):} Let us assume $g> id$. As $u$ is the strong unit, there exists the smallest $n_g>0$ such that $id<g\le n_gu$.

       \textbf{Subcase 1:} When $n_g=1$, $g\in [id, u]$. So, $g^*+g^*=u+g$. It follows that $$g=g^*+g^*-u=g^*+g^*-id^*-id^*=f(g^*-id^*).$$ 
       \textbf{Subcase 2:} If $n_g>1$, $g\not\le (n_g-1)u$. Define $$g_0:=g-(n_g-1)u.$$ By minimality of $n_g$, $g_0>id$ or $g_0\not\sim id$. Moreover, since $g\le n_gu$, we obtain $g_0\le u$. 
       
       For $g_0>id$, $g_0\in [id,u]$. So, $g_0=f(h_0)$ for some $h_0\in G$ using a similar argument as in subcase 1. It follows that $$g=g_0+(n_g-1)u=f(h_0)+(n_g-1) f(id^*)=f(h_0+(n_g-1)id^*).$$

       Now, for $g_0\not\sim id$, we have $g_0=g_0^++g_0^-$ by Proposition \ref{propdecomp}. Clearly, $g_0^+>id$. Moreover, $id<u$ and $g\le n_gu$ which implies $g_0\le u$. Therefore, $id<g_0^+\le u$. Consequently, $g_0^+= f(h_1)$ for some $h'\in G$ using subcase 1. Therefore, 
       $$g=f(h) \text{ for some } h\in G \text{ if and only if }g^-_0=f(h_0)\text{ for some }h_0\in G.$$ Note that, $g_0^-=g_0\wedge id$. So, 
       $$-g_0^-=-g_0\vee -id=-g_0\vee id.$$ 
       Clearly, $-g_0^->id$. Also, 
       $$-g_0=(n_g-1)u-g<(n_g-1)u\ \text{ and }\ id<u\le (n_g-1)u$$ which gives $id<-g_0^-\le (n_g-1) u$. Let $g_1:=-g_0^-$, then, $n_{g_1}<n_g$ and $$g=f(h)\text{ for some }h\in G\text{ if and only if }g_1=f(h_1)\text{ for some }h_1\in G.$$ Repeating this construction yields a sequence $g,g_1,g_2,\dots$ with $1\le n_{g_{i+1}}<n_{g_i}$. Hence the process terminates at some $g_k$ with $n_{g_k}=1$. Then, by subcase 1, $g_k=f(h_k)$ for some $h_k\in G$. Therefore, $g=f(h)$ for some $h\in G$.

       \textbf{Case (b):} If $g < id$, then $-g > id$. If $-g = f(h)$ for some $h \in G$, then $g = -f(h) = f(-h).$
Hence, surjectivity for positive elements implies surjectivity for negative ones.

       \textbf{Case (c):} If $g \not\sim id$, decompose $g$ as
\[
g = g^{+} + g^{-},
\quad \text{where} \quad
g^{+} = g \vee id \ \text{and} \ g^{-} = g \wedge id.
\]
Then $g^{+} > id$ and $g^{-} < id$.
If both $g^{+}$ and $g^{-}$ lie in the image of $f$, say
$g^{+} = f(h_{1})$ and $g^{-} = f(h_{2})$ for some $h_{1}, h_{2} \in G$,
then
\[
g = f(h_{1}) + f(h_{2})=h_1+h_1+h_2+h_2=f(h_1+h_2).
\]
    \end{enumerate}
    So, $f$ is an automorphism.

    \textbf{(3 $\implies$ 1)}: Now, if $f$ is an automorphism, then $f$ is also surjective. So, for each $g\in G$, there exists $h\in G$ such that $g=f(h)=h+h$. So, $G$ is $2$-divisible $\l u$-group.
\end{proof}
Note that since $f$ is an automorphism, it is injective. Consequently, for each
$g \in [id,u]$, the set
\[
X_g = \{ h \in G \mid h + h = u + g \}
\]
is a singleton. Hence, the function $^{*}$ is unique. That is, if
$\rho : [id,u] \to G$ is any function satisfying
\[
\rho(g) + \rho(g) = u + g \quad \text{for all } g \in [id,u],
\]
then $\rho = {}^{*}$. 
\begin{definition}\label{d4.3}
    Let $G$ be a $2$-divisible $\l u$-group with a strong unit $u$. For each $x,y\in [id,u]$,
    $$x\oplus y:=u\wedge (x+ y),\ \neg x:=u+ (-x),\ \kappa x:=x^*$$
    The structure $([id,u],\oplus,\neg,id, \kappa)$ will be denoted by $\Gamma(G,u,*)$
\end{definition}
\begin{proposition}\label{lem*preservesorder}
        If $g,h\in G$, then $g>h$ if and only if $g^*>h^*$.
\end{proposition}
    \begin{proof}
        First, assume that $g>h$. Then, $u+g>u+h$. So, $g^*+g^*>h^*+h^*$. Then, by Proposition \ref{propositionbirkoff}, $g^*>h^*$.

        For the converse, assume that $g^*>h^*$. Then, $u+g=g^*+g^*>h^*+g^*>h^*+h^*=u+h$. So, $g>h$.
    \end{proof}
\begin{theorem}
   $\Gamma(G,u,*)$ is a continuous algebra.
\end{theorem}
\begin{proof}
    For $\kappa:[id,u]\to [id,u]$ defined by $\kappa g=g^*$, we have to first check if it is well-defined.\\
    Let $g\in[id,u]$. Then, $g^*+g^*=u+g\ge u+id\ge id=id+id$. So, by Proposition \ref{propositionbirkoff}, $g^*\ge id$. Similarly, $g^*+g^*=u+g\le u+u$. So, $g^*\le u$. Therefore, $g^*\in [id, u]$.
    Now, $\Gamma(G,u,*)$ follows \ref{C1}-\ref{C6}, by \cite{CM}. We just have to check for \ref{C7}, \ref{C8}.
    For $g\in[id, u]$ we have
    \begin{align*}
        g\oplus \neg\kappa g&=u\wedge (g+u+(-g^*))\\
        &=u\wedge (g^*+g^*+(-g^*))\\
        &=u\wedge g^*\\
        &=g^*\\
        &=\kappa g.
    \end{align*} 
    Now, by Proposition \ref{lem*preservesorder}, we have that $u+(-g)+h\ge u$ if and only if $u+ (-g^*)+h^*\ge u$.
    Now, for $g,h\in G$,
    \begin{align*}
        (\neg\kappa g\oplus\kappa h)\oplus\neg(\neg g\oplus h)&=(u\wedge u+(-g^*)+h^*)\oplus (u+-(u\wedge u+(-g)+h))\\
        &= \begin{cases}
            u\oplus(u+(-u)) &\text{if}\ u+(-g)+h\ge u;\\
            (u+(-g^*)+h^*)\oplus (u+-(u+(-g)+h))&\text{if }u+(-g)+h<u.
        \end{cases}\\
        &=\begin{cases}
            u &\text{if}\ u+(-g)+h\ge u;\\
            u \wedge (u+(-g^*)+h^*+u+(-u)+g+(-h)) &\text{if}\ u+(-g)+h< u.
        \end{cases}\\
        &=\begin{cases}
            u &\text{if}\ u+(-g)+h\ge u;\\
            u \wedge (u+(-g^*)+h^*+g+(-h)) &\text{if}\ u+(-g)+h< u.
        \end{cases}
    \end{align*}
Assume \(u+(-g)+h<u\). Then \(u+(-g^{*})+h^{*}<u\), and hence \(h^{*}<g^{*}\). By monotonicity of addition,
$$h^{*}+g^{*}<g^{*}+g^{*}=u+g .$$
Adding \(h\) to both sides yields $$h^{*}+g^{*}+h<u+g+h,$$
which implies \(g^{*}+h<g+h^{*}\). Consequently, $$u<u+(-g^{*})+h^{*}+g+(-h).$$

    So, $(\neg\kappa g\oplus\kappa h)\oplus\neg(\neg g\oplus h)=u+ (-id)=\neg id$. Hence, $\Gamma(G,u,*)$ is a continuous algebra.
\end{proof}
On the other hand, we need to construct a $^*$ function for the Chang-group of a continuous algebra.
Let $\A=(A,\oplus,\neg,\id,\kappa)$ be a continuous algebra. Let $G_A$ denote the Chang's group for $(A,\oplus,\neg,\id)$, as defined in $\cite{CDM}$, with the class of $(\mathbf{x},\mathbf{y})$ denoted by $\langle\mathbf{x},\mathbf{y}\rangle$. Then, with strong unit $u=\langle(\nid),(\id)\rangle$, the interval $[id, u]=\set{\langle(a),(\id)\rangle|a\in A}$.
Now, define $^*:[id, u]\to G_A$ by $$\langle(a),(\id)\rangle^*=\langle(\kappa a),(\id)\rangle.$$
Then, we have $$\langle(a),(\id)\rangle^*+\langle(a),(\id)\rangle^*=\langle(\nid,a),(\id)\rangle=\langle(a),(\id)\rangle+\langle(\nid),(\id)\rangle.$$
By Theorem \ref{TFAElu*}, the Chang group of a continuous algebra is a $2$-divisible $\ell u$-group.

\section{Continuous terms and $\l u^*$-terms}
In this section, we introduce continuous terms and $\ell u^{*}$-terms, in a manner
analogous to the MV-terms and $\ell u$-terms defined in \cite{CM}. By
Theorem \ref{TFAElu*}, the existence of a unary operation $^{*}$ on $[\,id,u\,]$
is equivalent to the underlying $\ell u$-group being $2$-divisible. Accordingly,
throughout this section we work with $\ell u$-groups equipped with such a
$^{*}$-operation on $[\,id,u\,]$. We also establish several results that will be
used in the proof of the analogue of Chang’s completeness theorem for continuous
algebras.

Let $S_t:=\set{\id,\oplus,\neg,\kappa,(,),x_1,x_2,\dots,x_t}$. Then, a string on $S_t$ is a finite sequence of elements in $S_t$. 
\begin{definition}[Continuous terms]
    A string, $\tau$, on $S_t$ is said to be a continuous term if and only if there exists a finite sequence of strings, $\tau_1,\dots, \tau_n$ on $S_t$ such that $\tau_n=\tau$ and for each $i\le n$, one of the following holds
\begin{enumerate}[label=S\arabic*.]
    \item $\tau_i=x_i$;
    \item $\tau_i=\id$;
    \item $\tau_i=(\tau_j\oplus\tau_k)$ for some $j,k<i$;
    \item $\tau_i=\neg\tau_j$ for some $j<i$;
    \item $\tau_i=\kappa\tau_j$ for some $j<i$.
\end{enumerate}
The sequence is then called a formation sequence of $\tau$ and each $\tau_i$ is called a subterm of $\tau$.
\end{definition}
\begin{theorem}\label{t5.1}
    Let $\tau$ be a continuous term on $S_t$. Then $\tau$ satisfies exactly one of S1-S5. Moreover, $\tau_j$ (and $\tau_k$ in case of S3) in case of S3-S5 is unique. 
\end{theorem}
\begin{proof}
    By considering first term of a continuous term, $\tau$, it is clear that $\tau$ satisfies exactly one of 1.-5.
    Now, if $\tau=\neg\tau_1=\neg\tau_2$, then $\tau_1=\tau_2.$
    Similarly, if $\tau=\kappa\tau_1=\kappa\tau_2$, then $\tau_1=\tau_2.$
    Now, first note that no proper initial segment of a continuous term is a continuous term.
    We can prove this by strong induction on the length of a continuous term.
    If $\alpha$ is a continuous term of length $1$, then its proper initial segment has length $0$ and hence is not a term.
    Now, let no proper initial segment of a continuous term of length $< k$ be a continuous term.
    Then, let $\alpha$ be a continuous term of length $k$.
    If $\alpha=(\alpha_i\oplus \alpha_j)$, then any proper initial segment of $\alpha$ has unequal number of $($ and $)$ and hence, is not a continuous term.
    If $\alpha=\neg\alpha_i$, and $\beta=\neg\beta_i$ is a proper initial segment of $\alpha$ such that $\beta$ is a continuous term, then so is $\beta_i$.
    But then $\beta_i$ is a continuous term which is a proper initial segment of $\alpha_i$ of length $<k$.
    This contradicts the induction hypothesis.
    Therefore, $\alpha$ has no proper initial segment which is also a continuous term.
    Similar reasoning works for $\alpha=\kappa\alpha_i$.
    So, by strong induction, no proper segment of any continuous term is a continuous term.
    Now, let $\tau=(\tau_1\oplus \tau_2)=(\sigma_1\oplus \sigma_2)$.
    Then, $\tau_1\oplus\tau_2=\sigma_1\oplus \sigma_2$.
    If $\tau_1\neq \sigma_1$, then $\tau_1$ is a proper initial segment of $\sigma_1$ or vice versa.
    This is not possible, so $\tau_1=\sigma_1$.
    So, $\tau_2=\sigma_2$.
\end{proof}
Let $\A=(A,\oplus,\neg,\id,\kappa)$ be a continuous algebra. Let $\tau$ a continuous term on $S_t$. Let $a_1,\dots, a_n\in A$. Then $\tau^A(a_1,\dots,a_n)$ is the element in $A$ obtained by substituting $x_i$ by $a_i$ in $\tau$ and interpreting $\id,\oplus,\neg,\kappa$ as the corresponding operations in $\A$. This is well-defined by Theorem \ref{t5.1}

\begin{definition}
A continuous equation is a pair of continuous terms, written $\tau=\sigma$, where $\tau,\sigma$ are continuous terms.
\end{definition}
$\vDash$ is a relation between the class of continuous algebra and continuous equation such that $\A\vDash \tau=\sigma$ if and only if $\tau^A(a_1,\dots,a_n)=\sigma^A(a_1,\dots,a_n)$ for all $a_1,\dots,a_n\in A$, where $\A=(A,\oplus,\neg,\id,\kappa)$. If $\A\vDash \tau=\sigma$, then we say that $\A$ satisfies $\tau=\sigma$.
\begin{proposition}\label{l5.1}
    For each continuous algebra $\A$ and continuous terms $\tau,\sigma$, we have $\A\vDash \tau=\sigma$ if and only if $\A\vDash(\neg(\neg\tau\oplus\sigma)\oplus\neg(\neg\sigma\oplus\tau))=\id$.
\end{proposition}
This follows from Proposition \ref{propdistance0iffsamepoints}
\begin{theorem}\label{t5.2}
    An equation is satisfied by every continuous algebra if and only if it is satisfied by every continuous chain.
\end{theorem}
\begin{proof}
    One direction is trivial.
    Due to Proposition \ref{l5.1}, we can assume without the loss of generality, the equation is $\tau=\id$ for some continuous term $\tau$ on $S_n$.
    Let $\mathscr{C}\vDash\tau=\id$ for all continuous chains $\mathscr{C}$.
    Let $\A=(A,\oplus,\neg,\id,\kappa)$ be a continuous algebra.
    Then, if $\A$ is a subdirect product of a family of continuous chains, $\set{\A_i}_{i\in I}$ for some index set $I$, where $\A_i=(A_i,\oplus_i,\neg_i,\id_i,\kappa_i)$.
    Then, $\ex h:A\to\prod_{i\in I}A_i$ such that $h$ is an injective continuous homomorphism and $\pi_j\circ h$ is surjective for each $j\in I$.
    Let $a_1,\dots,a_n\in A$.
    If $\tau^A(a_1,\dots,a_n)\ne\id$, then
    \begin{align*}
        \tau^A(a_1,\dots,a_n)\ne\id &
        \implies h(\tau^A(a_1,\dots,a_n))\ne\mathbf{0}\\
        &\implies \pi_j\circ h(\tau^A(a_1,\dots,a_n))\ne\id_j\text{ for some }j\in I\\
        &\implies \A_j\not\vDash \tau=\id
    \end{align*}
    This contradicts our assumption that every chain satisfies $\tau=\id$.
    So, $\tau^A(a_1,\dots,a_n)=\id$ and hence $A\vDash\tau=\id$.
\end{proof}
\begin{definition}[$\l u^*$-terms]
    Let $T_t:=\set{0,+,-,^*,\wedge,\vee,(,),x_1,x_2,\dots,x_t}$. Then, a string on $T_t$ is a finite sequence of elements in $T_t$. A string, $\tau$, on $T_t$ is said to be an $\l u^*$-term if and only if there exists a finite sequence of strings, $\tau_1,\dots, \tau_n$ on $T_t$ such that $\tau_n=\tau$ and for each $i\le n$, one of the following holds
\begin{enumerate}[label=T\arabic*.]
    \item $\tau_i=x_i$;
    \item $\tau_i=0$;
    \item $\tau_i=(\tau_j+\tau_k)$ for some $j,k<i$;
    \item $\tau_i=(\tau_j\wedge\tau_k)$ for some $j,k<i$;
    \item $\tau_i=(\tau_j\vee\tau_k)$ for some $j,k<i$;
    \item $\tau_i=-\tau_j$ for some $j<i$;
    \item $\tau_i=^*\tau_j$ for some $j<i$.
\end{enumerate}
The sequence is then called a formation sequence of $\tau$ and each $\tau_i$ is called a subterm of $\tau$.
\end{definition}
\begin{theorem}
    Let $\tau$ be an $\l u^*$-term on $T_t$. Then $\tau$ satisfies exactly one of T1-T7 Moreover, $\tau_j$ (and $\tau_k$ in case of T3-T5) in case of T3-T7 is unique. 
\end{theorem}
\begin{proof}
    A proof similar to that of Theorem \ref{t5.1} proves this theorem.
\end{proof}
Let $G$ be a $2$-divisible $\l u$-group. Let $\tau$ an $\l u^*$ term on $T_{t}=\set{0,+,-,^*,\wedge,\vee,(,),x_1,x_2,\dots,x_t}$. Let $g_1,\dots, g_n\in G$. Then $\tau^G(g_1,\dots,g_n)$ is the element in $G$ obtained by substituting $x_i$ by $g_i$ in $\tau$ and interpreting $\id,+,-,\wedge,\vee,^*$ as the corresponding operations in $G$; where $(^*\tau)^G$ is interpreted as $f^{-1}(u+\tau^G)$ where $f:G\to G$ is the automorphism defined in Theorem \ref{TFAElu*}. In particular, if $\tau^G\in [id, u]$, then  $(^*\tau)^G=(\tau^G)^*$.

Now, define a map from continuous terms on $S_t$ to $\l u^*$-terms in $T_{t+1}=\set{0,+,-,^*,\wedge,\vee,(,),x_1,x_2,\dots,x_t,y}$ such that 
\begin{enumerate}
    \item $\id\mapsto 0$;
    \item $x_i\mapsto x_i$;
    \item $(\tau_i\oplus\tau_j)\mapsto(y\wedge(\hat\tau_i+\hat\tau_j))$;
    \item $\neg\tau_i\mapsto(y+-\hat\tau_i)$;
    \item $\kappa\tau_i\mapsto ^*(y+\hat\tau_i)$,
\end{enumerate}
where $\hat\tau$ is the image of $\tau$.
This map is well defined by unique readability theorem.
\begin{proposition}\label{l5.2}
    Let $\A=(A,\oplus,\neg,\id,\kappa)$ be a continuous algebra. Let $G_A$ be its Chang group. Also, denote the $\langle(a),(\id)\rangle$ by $[a]$. Let $\tau$ be a continuous term. Then, for every $a_1,\dots,a_n\in A$, $[\tau^A(a_1,\dots,a_n)]=\hat\tau^{G_A}([a_1]\dots,[a_n],[\nid])$.
\end{proposition}
\begin{proof}
    We can prove this by induction on number of operations in $\tau$.
    We just have to check for  $\kappa$ during induction, as the rest follows from MV-algebra.
    If $\tau=\kappa\sigma$ for some continuous term $\kappa$, then $\sigma$ has lesser number of operations than $\tau$.
    So, by induction hypothesis, $\hat\sigma^{G_A}([a_1],\dots,[a_n],[\nid])=[\sigma^A(a_1,\dots,a_n)]$.
    So, we have
    \begin{align*}
        \hat\tau^{G_A}([a_1],\dots,[a_n],[\nid])&=\hat\sigma^{G_A}([a_1],\dots,[a_n],[\nid])^*\\
        &=[\sigma^A(a_1,\dots,a_n))]^*\\
        &=[\kappa\sigma^A(a_1,\dots,a_n)]\\
        &=[\tau^A(a_1,\dots,a_n)].
    \end{align*}
\end{proof}\label{l5.3}
\begin{proposition}
    Let $G$ be an $2$-divisible $\l u$-group with strong unit $u$. Let $\A=\Gamma(G, u,*)$. Let $\tau$ be a continuous term on $S_n$. Let $g_1,\dots,g_n\in [id, u]\subseteq G$. Then, $\tau^{[id, u]}(g_1,\dots,g_n)=\hat\tau^{G}(g_1,\dots,g_n,u)$.
\end{proposition}
\begin{proof}
    We can prove this by induction on number of operations in $\tau$.
    We just have to check for  $\kappa$ during induction, as the rest follows from MV-algebra.
    If $\tau=\kappa\sigma$ for some continuous term $\kappa$, then $\sigma$ has lesser number of operations than $\tau$.
    So, by induction hypothesis, we have
    \begin{align*}
        \sigma^{[id, u]}(g_1,\dots,g_n)&=\hat\sigma^G(g_1,\dots,g_n,u)\\
        \implies \tau^{[id, u]}(g_1,\dots,g_n)&=\kappa\sigma^{[id,u]}(g_1,\dots,g_n)\\
        &=\kappa\hat\sigma^G(g_1,\dots,g_n,u)\\
        &=\hat\sigma^G(g_1,\dots,g_n,u)^*\\
        &=\hat\tau^G(g_1,\dots,g_n,u).
    \end{align*}
\end{proof}
\section{Weak Completeness Theorem For Continuous Logic}
In this section, we prove the analogue for Chang's completeness theorem for continuous algebra. Then we use it to prove weak completeness for continuous propositional logic algebraically, which is the main goal of this article. 
\begin{theorem}
    A continuous equation is satisfied by every continuous algebra if and only if it is satisfied by $[0,1]$.
\end{theorem}
\begin{proof}
     Suppose an equation on $S_n$ fails in a continuous algebra $\A=(A,\oplus,\neg,\id,\kappa)$. 
     Without the loss of generality, let the equation be $\tau=\id$. We can assume this  due to Proposition \ref{l5.1}.
     Without the loss of generality, let $\A$ be a continuous chain. We can assume this since if an equation is not satisfied by a continuous algebra, then there exists some continuous chain which does not satisfy that equation, by Theorem \ref{t5.2}.
     Then, there exist \( a_1, \dots, a_n \in A \) such that \( \id<\tau^A(a_1, \dots, a_n) \le \nid \).
     Let \( G_A \) denote the Chang-group of \( \A \).
     Then, we have that 
     $[\id] \prec [\tau^A(a_1,\dots,a_n))]\preceq [\nid].$
     By Proposition \ref{l5.2}, we get
     \[
    [\id] \prec \hat{\tau}^{G_A}([a_1],\dots,[a_n],[\nid]) \preceq [\nid].
    \]  
    Let $S$ be the subgroup of \( G_A \) generated by the elements \( [a_1],\dots,[a_n],[\nid] \). Then, $S$ has an ordering, induced by $G_A$.

    \smallskip
    Note that $G_A$ is an $\l$-group and hence a torsion free abelian group. Since $S$ is a subgroup of $G_A$, it is also a torsion free abelian group. Also, $S$ is finitely generated.
    So, by the fundamental theorem on torsion-free abelian groups, there exists $\Phi:S\to \mathds{Z}^r$, a group isomorphism between $\mathds{Z}^r$ and $S$ for some integer \( r \geq 1 \).
    Let $P_0:=\set{x\in S| [\id]\preceq x}$. Let $P:=\Phi(P_0)$.
    Note that since $\Phi$ is an isomorphism and the ordering on $G_A$ is total, 
    \begin{equation}\label{P-P}
        P \cap -P = \{0\} \quad \text{and} \quad P \cup -P = \mathds{Z}^r
    \end{equation}
    Now, define a relation, $\le_P$, on $\mathds{Z}^r$ by \(\fall  \mathbf{h}, \mathbf{k} \in \mathds{Z}^r \), 
    \[\mathbf{h} \leq_P \mathbf{k} \quad \text{iff} \quad \mathbf{k} - \mathbf{h} \in P.\]
    By the definition of $\le_P$, and \ref{P-P}, $\le_P$ is a total ordering on $\mathds{Z}^r$.
    \smallskip
    
    Let \( \tau_1,\tau_2, \dots, \tau_t \) be a formation sequence of $\tau$. Without the loss of generality, let $\tau_i=x_i$ for $i\le n$. Also, let $\id$ be a subterm of $\tau$.
    Now, we create a formation sequence of $\hat{\tau}$ using the formation sequence of $\tau$ as follows:
    \begin{enumerate}
        \item Let $\sigma_0=y$.
        \item If $\tau_i=x_i$, let $\sigma_{2i-1}=0$ and $\sigma_{2i}=x_i$.
        \item If $\tau_i=\id$, let $\sigma_{2i-1}=0$ and $\sigma_{2i}=0$.
        \item If $\tau_i=(\tau_j\oplus \tau_k)$, let $\sigma_{2i-1}=(\sigma_{2j}+\sigma_{2k})$ and $\sigma_{2i}=(y\wedge (\sigma_{2j}+\sigma_{2k}))$.
        \item If $\tau_i=\neg\tau_j$, let $\sigma_{2i-1}=-\sigma_{2j}$ and $\sigma_{2i}=(y + -\sigma_{2j})$.
        \item If $\tau_i=\kappa\tau_j$, let $\sigma_{2i-1}=y+\sigma_{2j}$ and $\sigma_{2i}=^*(y + \sigma_{2j})$.
    \end{enumerate}
    \begin{claim} $\hat\tau_k=\sigma_{2k}$ for all $1\le k\le t$.
    \end{claim}
    
        For $k=1$, $\hat\tau_k=\sigma_{2k}=x_k$.
        Let $\hat\tau_{k}=\sigma_{2k}$ for $1\le k<i\le t$.
        Then, if $\tau_i=x_i$, then $\hat\tau_i=\sigma_{2i}=x_i$.
        If $\tau_i=\id$, then $\hat\tau_i=\sigma_{2i}=0$.
        If $\tau_i=(\tau_j\oplus \tau_k)$, then $\hat\tau_i=(y\wedge(\hat\tau_j+\hat\tau_k))=(y\wedge(\sigma_{2j}+\sigma_{2k}))=\sigma_{2i}$.
        If $\tau_i=\neg\tau_j$, then $\hat\tau_i=(y+-\hat\tau_j))=(y+-\sigma_{2j})=\sigma_{2i}$.
        If $\tau_i=(\tau_j\oplus \tau_k)$, then $\hat\tau_i=^*(y+\hat\tau_j)=^*(y+\sigma_{2i})=\sigma_{2i}$.
        So, by strong induction, $\hat\tau_i=\sigma_{2i}$ for all $1\le i\le t$.

    Let  $\Phi(2^{t+1}[\nid])=\mathbf{h}_0$ and $\Phi(2^{t+1}[a_i])=\mathbf{h}_i$ for $1\le i\le n$. Now, consider the map $v:\set{\sigma_0,\sigma_{2},\sigma_4,\dots, \sigma_{2n}}\to \mathds{Z}^r$ defined by $\sigma_{2i}\mapsto \mathbf{h}_i$.
    Then, $v$ can be extended to $\hat{v}:\set{\sigma_{2i}|0\le i\le t}\to \mathds{Z}^r$ such that 
    \begin{enumerate}
        \item $\hat{v}(\sigma_{2i})=v(\sigma_{2i})$ for $i\le n$.
        \item $\hat{v}(0)=\mathbf{0}=(0,0,\dots,0)$.
        \item $\hat{v(}(y\wedge (\sigma_{2i}+\sigma_{2j})))=\mathbf{h}_0\wedge_P (\hat{v}(\sigma_{2i})+\hat{v}(\sigma_{2j}))$.
        \item $\hat{v}((y + -\sigma_{2i}))=\mathbf{h}_0-\hat{v}(\sigma_{2i})$.
        \item $\hat{v}(^*(y + \sigma_{2i}))=\dfrac{1}{2}(\mathbf{h}_0+\hat{v}(\sigma_{2i}))$.
    \end{enumerate}
    By unique readability and formation of $\sigma_{2i}$, this extension is well-defined. Note that since formation sequence of $\tau$ has $t$ terms and $2^{t+1}$ divides $\textbf{h}_i$ for each $i=0,\dots,n$, $\hat{v}(^*(y + \sigma_{2i}))\in \mathds{Z}^r$ as we are multiplying by half for each $\kappa$ which may appear at most $t-1$ many times in the formation sequence of $\tau$.
    Denote $\hat{v}(\sigma_{2i})$ by $\mathbf{h}_i$ for $i\le t$.
    Let $T$ be the totally ordered group $(\mathds{Z}^r,\le_P)$.
    \begin{claim}
$\mathbf{0}\le_P\mathbf{h}_i\le_P\mathbf{h}_0$ for $0\le i\le t$.
\end{claim}
        
            We can prove  this by induction on $i$.
            If $i=0$, then $\mathbf{h}_0=\Phi(2^{t+1}[\nid])$.
            So, $$\mathbf{h}_0-\mathbf{0}=\Phi(2^{t+1}[\nid])-\Phi([\id])=\Phi(2^{t+1}[\nid]).$$
            Since $[\id]\preceq2^{t+1}[\nid]$, $\mathbf{h}_0-\mathbf{0}\in P$.
            Hence, $$\mathbf{0}\le_P\mathbf{h}_0\le_P\mathbf{h}_0.$$
            Now, let $\mathbf{0}\le_P\mathbf{h}_i\le_P\mathbf{h}_0$ for all $i< k$.
            Then, if $\tau_k=x_i$ we have $$\mathbf{h}_i=\Phi(2^{t+1}[a_i]).$$
            So, by similar reasoning as above, $\mathbf{0}\le_P\mathbf{h}_k.$
            Also, $\mathbf{h}_0-\mathbf{h}_k=\Phi(2^{t+1}(\nid-[a_k]))$.
            So, $\mathbf{h}_0\ge_P \mathbf{h}_k$.
            If $\tau_k=\id$, then $\mathbf{h}_k=\mathbf{0}\le_P\mathbf{h}_0$.
            If $\tau_k=(\tau_i\oplus \tau_j)$, then $\mathbf{h}_k=\mathbf{h}_0\wedge_P(\mathbf{h}_i+\mathbf{h}_j)$, where $i,j< k$.
            So, by induction hypothesis, $\mathbf{0}\le_P\mathbf{h}_k$ and by definition, $\mathbf{h}_k\le_P\mathbf{h}_0$.
            If $\tau_k=\neg\tau_i$, then $\mathbf{h}_k=\mathbf{h}_0-\mathbf{h}_i$ for some $i<k$.
            Since $\mathbf{h}_i\le_P\mathbf{h}_0,\ \mathbf{h}_0-\mathbf{h}_i\in P$.
            So, $\mathbf{0}\le_p\mathbf{h}_k$.
            Since $\mathbf{0}\le_P\mathbf{h}_i$, $\mathbf{h}_i\in P$.
            Therefore $\mathbf{h}_0-\mathbf{h}_0+\mathbf{h}_i\in P$ gives $\mathbf{h}_k\le_P\mathbf{h}_0$.
            If $\tau_k=\kappa\tau_i$, then $\mathbf{h}_k=\half(\mathbf{h}_0+\mathbf{h}_i)$ for some $i< k$.
            Then, $\mathbf{0}\le_P\mathbf{h}_k\le\mathbf{h}_0$.
       
        \begin{claim} 
        $\mathbf{h}_i=\Phi(2^t\hat\tau_i^{G_A}([a_1],\dots,[a_n],[\nid]))$ for all $1\le i\le t$.
        \end{claim}
        
            We can prove  this by induction on $i$.
            If $i=1$, then $\tau_i=x_i$.
            So, $\hat\tau=x_i$.
            Therefore, $$\mathbf{h}_i=\Phi(2^t\hat\tau^{G_A}_i([a_1],\dots,[a_n],[\nid])).$$
            Now, let $$\mathbf{h}_i=\Phi(2^t\hat\tau^{G_A}_i([a_1],\dots,[a_n],[\nid]))\ \text{for all}\ i< k.$$
            Then, if $\tau_k=x_i$, then by similar reasoning as above, $\mathbf{h}_k=\Phi(2^t\hat\tau^{G_A}_i([a_1],\dots,[a_n],[\nid]))$.
            If $\tau_k=\id$, then $\hat\tau=0$.
            So, $\mathbf{h}_k=\mathbf{0}=\Phi(2^t\hat\tau^{G_A}_k([a_1],\dots,[a_n],[\nid]))$.
            If $\tau_k=(\tau_i\oplus \tau_j)$, then $\mathbf{h}_k=\mathbf{h}_0\wedge_P(\mathbf{h}_i+\mathbf{h}_j)$, where $i,j< k$ and $\hat\tau_k=(y\wedge(\hat\tau_i+\hat\tau_j))$.
            Now, if $$\nid\preceq\hat\tau_i([a_1],\dots,[a_n],[\nid])+\hat\tau_j([a_1],\dots,[a_n],[\nid]),$$ then we have $$\Phi(\hat\tau_i([a_1],\dots,[a_n],[\nid])+\hat\tau_j([a_1],\dots,[a_n],[\nid])-\nid)\in P.$$
            So, $\Phi(2^t\hat\tau_i([a_1],\dots,[a_n],[\nid])+2^t\hat\tau_j([a_1],\dots,[a_n],[\nid])-2^t\nid)\in P$.
            Hence $$\mathbf{h}_i+\mathbf{h}_j-\mathbf{h}_0\in P.$$
            So, $\mathbf{h}_k=\mathbf{h}_0$.
            Else, we can similarly prove that $\mathbf{h}_k=\mathbf{h}_i+\mathbf{h}_j$.
            Hence $$\mathbf{h}_k= \Phi(2^t\hat\tau^{G_A}_k([a_1],\dots,[a_n],[\nid])).$$
            If $\tau_k=\neg \tau_i$, then $\mathbf{h}_k=\mathbf{h}_0-\mathbf{h}_i$ and $\hat\tau_k=y-\hat\tau_i$.
            Consequently $$\mathbf{h}_k=\Phi(2^t\hat\tau^{G_A}_k([a_1],\dots,[a_n],[\nid])).$$
            If $\tau_k=\kappa\tau_i$, then $\mathbf{h}_k=\dfrac{1}{2}(\mathbf{h}_0+\mathbf{h}_i)$ and $\hat{\tau}_k=^*\tau_i$.
            Then, 
            \begin{align*}
                &2\Phi(2^t\hat\tau^{G_A}_k([a_1],\dots,[a_n],[\nid]))\\
                =&\Phi(2^t\hat\tau^{G_A}_i([a_1],\dots,[a_n],[\nid])^*)+\Phi(2^t\hat\tau^{G_A}_i([a_1],\dots,[a_n],[\nid])^*)\\
            =&\Phi(2^t(\hat\tau^{G_A}_i([a_1],\dots,[a_n],[\nid]))^*+(\hat\tau^{G_A}_i([a_1],\dots,[a_n],[\nid]))^*)\\
                =&\Phi(2^t(\nid+\hat\tau^{G_A}_i([a_1],\dots,[a_n],[\nid])))\\
                =&2\mathbf{h}_k
                \end{align*}
                Hence $\Phi(2^t\hat\tau^{G_A}_k([a_1],\dots,[a_n],[\nid]))=\mathbf{h}_k.$
        
    In particular, we have
$$\mathbf{0} \leq_P \mathbf{h}_1, \dots, \mathbf{h}_n \leq_P\mathbf{h}_0, \quad 0 \leq_P \mathbf{h}_t \leq_P \mathbf{h}_0, \quad \mathbf{0} \neq \mathbf{h}_t = \Phi(2^t\widehat{\tau}^{G_A}([a_1],\dots,[a_n],[\nid])). $$ 
    Let \( \omega \) be a permutation of \( \{0, \dots, t\} \) such that $\mathbf{h}_{\omega(0)} \leq_P \mathbf{h}_{\omega(1)} \leq_P \dots \leq_P \mathbf{h}_{\omega(t)}.$    
    For each \( j = 1, \dots, t \), define a $\mathbf{d}_j:= \mathbf{h}_{\omega(j)} - \mathbf{h}_{\omega(j-1)}\in \mathds{Z}^r$.
    Let $P^* = \left\{\sum_{j=1}^t \lambda_j \mathbf{d}_j \,\middle|\, 0 \leq \lambda_j \in \mathds{R} \right\}$.
    Then, \( P^* \) is a closed and convex subset of \( \mathds{R}^r \).
    Note that if for some $0\le i,j\le t$ we have $\mathbf{h}_i\le_P\mathbf{h}_j$, then $\omega^{-1}(i)<\omega^{-1}(j)$. 
    So, we have
    \begin{align*}
        &\mathbf{h}_i\le_P\mathbf{h}_j\\
        \implies& \mathbf{h}_{\omega(\omega^{-1}(i))}\le_P\mathbf{h}_{\omega(\omega^{-1}(j))}\\
        \implies& \mathbf{h}_j-\mathbf{h}_i=\mathbf{h}_{\omega(\omega^{-1}(j))}-\mathbf{h}_{\omega(\omega^{-1}(j)-1)}+\mathbf{h}_{\omega(\omega^{-1}(j)+1)}-\mathbf{h}_{\omega(\omega^{-1}(j)-2)}\dots-\mathbf{h}_{\omega(\omega^{-1}(i))}\\
        \implies& \mathbf{h}_j-\mathbf{h}_i=\sum_{k=1}^t\lambda_k\mathbf{d}_k,
    \end{align*}
    where $\lambda_k=\begin{cases}
        1&\text{ if }w^{-1}(i)<w^{-1}(j);\\
        0&\text{ otherwise.}
    \end{cases}$\\
    So, $\mathbf{h}_j-\mathbf{h}_i\in P^*$.
    Similarly, we get that if $\mathbf{h}_j<_P\mathbf{h}_i$, then $\mathbf{h}_j-\mathbf{h}_i\notin P^*$.
    So, $\mathbf{h}_j\ge_P \mathbf{h}_i$ if and only if $\mathbf{h}_j-\mathbf{h}_i\in P^*$.

    Also, note that, if \( \lambda_1, \dots, \lambda_t \in\mathds{R}\) such that \(\lambda_i \geq 0 \) and \( \sum_{i=1}^t \lambda_i \mathbf{d}_i = 0 \), then \( \lambda_i = 0\  \fall i \) such that \( \mathbf{d}_i \neq 0 \).
    This follows by a similar proof as in \cite{CDM}.
    So, there exists $\mathbf{g} = (\gamma_1, \dots, \gamma_r) \in \mathds{R}^r$ such that $ \mathbf{g} \cdot \mathbf{d}_j > 0 $ for all $ \mathbf{d}_j\ne0, \, j = 1, \ldots, t $. 
    Since dot product is continuous, there exists an $\epsilon$-ball around $\mathbf{g}$ such that $\mathbf{g}\cdot\mathbf{d}_j>0$ for all $\mathbf{d}_j\neq0$ for some $\epsilon>0$.
    So, without loss of generality, let $ \gamma_1, \ldots, \gamma_r $ be linearly independent over $ \mathds{Q} $.
    
    Now, let $ \pi^+_\mathbf{g} := \{ \mathbf{h} \in \mathds{R}^r \mid \mathbf{h} \cdot \mathbf{g} \geq 0 \}$ and let $P' := \pi^+_\mathbf{g} \cap \mathds{Z}^r$.\\
    Then, we have $P^*\subseteq \pi^+_{\mathbf{g}}$.
    Also, $P'\cup -P'=\mathds{Z}^r$.
    Now, let $\mathbf{x}=(x_1,\dots,x_r)\in P'\cap -P'$.
    Then, we have $x_i\in \mathds{Z}$ and 
$$\mathbf{x}\cdot \mathbf{g}\ge0\ \text{and} \ \mathbf{x}\cdot\mathbf{g}\le0
        \implies \mathbf{x}\cdot\mathbf{g}=0
        \implies x_1\gamma_1+x_2\gamma_2+\dots+x_r\gamma_r=0
        \implies x_i=0.$$
    Since $\gamma_i$-s are linearly independent over $\mathds{Q}$.
    So, $P'\cap-P'=\set{\mathbf{0}}$.
    Now, define a relation, $\le_{P'}$, on $\mathds{Z}^r$ defined by $\mathbf{h}\le_{P'}\mathbf{k}$ if and only if $\mathbf{k}-\mathbf{h}\in P'$.
    Then, $T'=(\mathds{Z}^r, \le_{P'})$ is a totally ordered abelian group.
    Note that if $0\le i,j\le t$, then we have
$$  \mathbf{h}_i\le_P\mathbf{h}_j
        \implies \mathbf{h}_j-\mathbf{h}_i\in P^*\cap P\subseteq \pi^+_{\mathbf{g}}\cap \mathds{Z}^r=P'.$$
Hence $\mathbf{h}_i\le_{P'}\mathbf{h}_j$.
    Similarly, we get $\mathbf{h}_i\le_P \mathbf{h}_j\iff \mathbf{h}_i\le_{P'}\mathbf{h}_j$.
    Then, we have
    \begin{enumerate}
        \item $\mathbf{0}\le_{P'}\mathbf{h}_i\le_{P'}\mathbf{h}_0$ for $0\le i\le t$. This immediately follows from a previous claim.
        \item $\mathbf{h}_i=\Phi(2^t\hat\tau_i^{G_A}([a_1],\dots,[a_n],[\nid]))$.
    \end{enumerate}
    In particular, we have
    \begin{align}
        \mathbf{0} &\leq_{P'} \mathbf{h}_1, \dots, \mathbf{h}_n \leq_{P'}\mathbf{h}_0,\label{ineq1}\\
        \quad \mathbf{0} &\leq_{P'} \mathbf{h}_t \leq_{P'} \mathbf{h}_0, \label{ineq2}\\
        \mathbf{0} &\neq \mathbf{h}_t = \Phi(2^t\widehat{\tau}^{G_A}([a_1],\dots,[a_n],[\nid])). \label{ineq3}
    \end{align}
    Now, consider the map $\theta:\mathds{Z}^r\to \mathds{R}$ given by $$\theta(\mathbf{h})= \frac{\mathbf{g}\cdot\mathbf{h}}{\mathbf{g}\cdot\mathbf{h}_0}.$$
    Since $\gamma_1,\dots,\gamma_r$ are linearly independent over $\mathds{Q}$, $\theta$ is an injective group homomorphism between $Z^r$ and $\mathds{R}$.
    \begin{claim}
        $\theta$ preserves the inequalities in \ref{ineq1}, \ref{ineq2}, \ref{ineq3}.
    \end{claim}
    
        If $\mathbf{a}\le_{P'}\mathbf{b}$, 
        \begin{align*}
            \mathbf{b}-\mathbf{a}\in P'&\implies(\mathbf{b}-\mathbf{a})\cdot \mathbf{g}\ge0\\
            &\implies\mathbf{b}\cdot \mathbf{g}-\mathbf{a}\cdot\mathbf{g}\ge 0\\
&\implies\mathbf{b}\cdot\mathbf{g}\ge\mathbf{a}\cdot\mathbf{g}\\
            &\implies \frac{\mathbf{g}\cdot\mathbf{b}}{\mathbf{g}\cdot\mathbf{h}_0}\ge\frac{\mathbf{g}\cdot\mathbf{a}}{\mathbf{g}\cdot\mathbf{h}_0}\\
            &\implies\theta(\mathbf{b})\ge\theta(\mathbf{a})
        \end{align*}

    Now, let $\delta_i := \theta(\mathbf{h}_i)$ for $0\le i\le t$.
    Then, we have $0\le\delta_i\le\delta_0 $ for $0\le i\le t$ and $ \delta_0 = 1 $.
    \begin{claim}
        $\delta_i=\hat\tau^{\mathds{R}}_i(\delta_1,\dots,\delta_n,1)$ for $1\le i\le t$.
    \end{claim}

        We prove this by induction on $i$.
        If $i=1$, then $\tau_i=x_1$.
        So, $\delta_1=\hat\tau^{\mathds{R}}_i(\delta_1,\dots,\delta_n,1)$.
        Now, let $\delta_i=\hat\tau^{\mathds{R}}_i(\delta_1,\dots,\delta_n,1)$ for  all $i<k$.
        Then, if $\tau_k=x_i$, by similar reasoning as above, $\delta_k=\hat\tau^{\mathds{R}}_k(\delta_1,\dots,\delta_n,1)$.
        If $\tau_k=\id$, then $\hat\tau_k=0$ and $\mathbf{h}_k=\mathbf{0}$.
        So, $\delta=0=\hat\tau^{\mathds{R}}_k(\delta_1,\dots,\delta_n,1)$.
        If $\tau_k=(\tau_i\oplus\tau_j)$, then $\hat{\tau}_k=(y\wedge(\hat{\tau}_i+\hat{\tau}_j))$.
        Then, $\mathbf{h}_k=\mathbf{h}_0\wedge_P(\mathbf{h}_i+\mathbf{h}_j)=\mathbf{h}_0\wedge_{P'}(\mathbf{h}_i+\mathbf{h}_j)$.
        Now, if $\mathbf{h}_i-\mathbf{h}_j\le_{P'}\mathbf{h}_0$, then $\theta(\mathbf{h}_i)+\theta(\mathbf{h}_j)\le\theta(\mathbf{h}_0)=1$.
        So, $\delta_i+\delta_j\le1$.
        So, $\delta_k=\delta_i+\delta_j$.
        Else, we can similarly prove that $\delta_k=1$.
        So, $\delta_k=\hat\tau^{\mathds{R}}_i(\delta_1,\dots,\delta_n,1)$.
        If $\tau_k=\neg\tau_i$, then $\mathbf{h}_k=\mathbf{h}_0-\mathbf{h}_i$ and $\hat{\tau}_k=y-\hat{\tau}_i$.
        So, $\delta_k=\hat\tau^{\mathds{R}}_i(\delta_1,\dots,\delta_n,1)$.
        If $\tau_k=\kappa\tau_i$, then $\mathbf{h}_k=\dfrac{1}{2}(\mathbf{h}_0+\mathbf{h}_i)$ and $\hat{\tau}_k=^*\tau_i$.
        So, $\delta_k=\hat\tau^{\mathds{R}}_i(\delta_1,\dots,\delta_n,1)$.
    Then, by Proposition \ref{l5.3}, we get that if $\mathscr{R}=\Gamma(\mathds{R},1,*)=([0,1],\oplus,\neg,0,\kappa)$, then $\mathscr{R}\not\vDash \tau=\id$.
\end{proof}
\begin{theorem}
    Every tautology in CPL is provable, i.e. if $\vDash \alpha$ for some $\alpha\in \mathscr{L}$, then $\vdash\alpha$.
\end{theorem}
\begin{proof}
    Let $\alpha\in \mathscr{L}$ be not provable.
    Let the propositional variables in $\alpha$ be $P_1,\dots,P_n$.
    Let $\tau_{\alpha}$ be a continuous term obtained by substituting $\itai$ by $\oplus\neg$ and $P_i$ by $x_i$ in $\alpha$.
    Let $\mathcal{C}=(\mathscr{L}/_{\equiv},\oplus,\neg,\bot,\half)$
    \begin{claim}$[\alpha]=\tau_{\alpha}^{\mathscr{L}/_{\equiv}}([P_1],\dots,[P_n])$.
    \end{claim}
    
        We can prove this by induction on the number of operations in $\alpha$.
        If ${\alpha}$ has no operations, then $\alpha=P_1$.
        Then, $\tau=x_i$ and $[\alpha]=\tau^{\mathscr{L}/_{\equiv}}(P_1)$.
        So, let $[\varphi]=\tau_{\varphi}^{\mathscr{L}/_{\equiv}}([P_1],\dots,[P_n])$ for every $\varphi\in\mathscr{L}/_{\equiv}$ with less than $k$ operations.
        If ${\alpha}$ has $k$ operations, then ${\alpha}={\beta}\itai{\gamma}$ or $\neg{\beta}$ or $\kappa{\beta}$ for some $\beta,\gamma\in \mathscr{L}/_{\equiv}$.
        If ${\alpha}={\beta}\itai{\gamma}$, then 
        \begin{align*}
            \tau_{\alpha}^{\mathscr{L}/_{\equiv}}(P_1,\dots,P_n)&=(\tau_\beta^{\mathscr{L/\equiv}}(P_1,\dots,P_n)\oplus\neg\tau_\gamma^{\mathscr{L/\equiv}}(P_1,\dots,P_n))\\
            &=[\beta]\oplus\neg[\gamma]\\
            &=[\beta\oplus\neg\gamma]\\
            &=[\alpha].
        \end{align*}
        The other cases follow through a similar reasoning. 
        
        Since $\alpha$ is not provable, $[\alpha]\neq \neg\bot$. So, $\mathcal{C}\not\vDash \tau_{\alpha}=\neg\id$ as $\id$ in $\mathcal{C}$ is $\bot$.
    So, $\tau_{\alpha}=\nid$ is not  satisfied by $([0,1],\oplus,\neg,0 ,\kappa)$. Let $a_1,\dots,a_n\in [0,1]$ such that $\tau^{[0,1]}(a_1,\dots,a_n)\neq 1$. Since $[1,0]=\overline{[0,1]}$ is isomorphic to $([0,1],\oplus,\neg,0,\kappa)$ by the isomorphism $x\mapsto 1-x$.
    So, $\tau_{\alpha}^{[1,0]}(1-a_1,\dots,1-a_n)\neq 0$.
    Consider the valuation $v:\mathscr{L_0}\to [0,1]$ such that $P_i\mapsto 1-a_i$.
    Then,  $\not\vDash\alpha$ as $\hat{v}(\alpha)\ne 0$.
    \end{proof}

\section*{Conclusion}
In this paper, we have developed algebraic semantics for the propositional fragment of continuous logic by introducing the notion of \emph{continuous algebra}, which may be viewed as a Lindenbaum-type algebra for Continuous Propositional Logic. By endowing MV-algebras with a unary operator $\kappa$ corresponding to the halving connective and by formulating suitable axioms governing its behavior, we obtain an algebraic framework that preserves the essential structure of MV-algebras while faithfully capturing the semantic features specific to continuous logic.

We have carried out a detailed structural analysis of continuous algebras, including the study of ideals, quotient constructions, and homomorphisms, and we have shown that every continuous algebra can be represented as a subdirect product of totally ordered continuous algebras, called continuous chains. Moreover, by establishing a precise correspondence between continuous algebras and 2-divisible lattice-ordered groups with strong unit, we have extended Chang completeness theorem to this setting and derived a purely algebraic proof of the weak completeness of Continuous Propositional Logic.

The algebraic framework developed here opens several directions for further research. In ongoing work, we aim to provide a purely algebraic proof of strong approximated completeness for CPL. Beyond completeness, we plan to investigate additional structural properties of continuous algebras, including the amalgamation proposition and its logical counterpart, the deductive interpolation proposition for continuous logic. We also intend to extend the present approach to richer logical systems by developing algebraic semantics for continuous first-order logic, drawing inspiration from Tarski’s theory of cylindric algebras \cite{LH}, while at the same time proposing a notion of \emph{complete continuous algebra} suited to the continuous setting. Finally, we plan to explore algebraic structures corresponding to affine logic \cite{BI}, thereby broadening the scope of the algebraic methodology introduced in this work.

\section{Acknowledgment:}
The authors thank the reviewers for their careful reading and valuable comments, which led to a significant improvement of the present version of the paper. They are grateful to Dr. Antonio Di Nola and Dr. Anand Pillay for introducing them to the areas of MV-algebras and continuous logic. They also thank Dr. Brian Wynne and Dr. Katsuhiko Sano for helpful comments and valuable remarks.

\section{Declarations}
\begin{itemize}
    \item \textbf{Author Contributions}:- All authors contributed equally.
    \item \textbf{Funding}:- Not applicable.
    \item \textbf{Availability of data and materials}:- Not applicable.
    \item \textbf{Conflict of interest}:- The authors declare no competing interests.
    \item \textbf{Ethical approval}:- Not applicable.
\end{itemize}

\end{document}